\documentclass[a4paper,12pt]{amsart}
\usepackage[latin1]{inputenc}
\usepackage{amsmath}
\usepackage{amsfonts}
\usepackage{amssymb}
\usepackage[shortlabels]{enumitem}
\usepackage{amsthm}
\newcommand*{\rom}[1]{\expandafter\@slowromancap\romannumeral #1@}

\usepackage{fourier} 
\usepackage{hyperref} 
\usepackage{xcolor} 
\usepackage[english]{babel}
\usepackage{float}
\usepackage{cleveref}
\usepackage{verbatim} 
\DeclareMathOperator*{\foo}{\scalerel*{+}{\sum}}

\usepackage{scalerel}
\crefformat{section}{\S#2#1#3} 
\crefformat{subsection}{\S#2#1#3}
\crefformat{subsubsection}{\S#2#1#3}

\newtheorem{theorem}{Theorem}[section]
\newtheorem{lemma}[theorem]{Lemma}
\newtheorem{prop}[theorem]{Proposition}
\newtheorem{corollary}[theorem]{Corollary}
\theoremstyle{definition}
\newtheorem{definition}[theorem]{Definition}
\newtheorem{example}[theorem]{Example}

\theoremstyle{remark}
\newtheorem{remark}[theorem]{Remark}

\newtheorem*{ques*}{Question}

\newcommand{\R}{\mathbb{R}}

\newcommand{\N}{\mathbb{N}}
\newcommand{\F}{\mathcal{F}}

\numberwithin{equation}{section}
\newcommand{\C}{\mathbb{C}}

\newcommand{\vphi}{\varphi}
\newcommand{\Fock}{\mathcal{F}({H})}

\newcommand{\bea}{\textcolor{black}}
\newcommand{\RomanNumeralCaps}[1]
    {\MakeUppercase{\romannumeral #1}}
\definecolor{darkblue}{rgb}{0.0, 0.0, 0.55} 

\title[SHARP DECAY ESTIMATES FOR FOKKER-PLANCK EQUATIONS]{PROPAGATOR NORM AND SHARP DECAY ESTIMATES FOR FOKKER-PLANCK EQUATIONS WITH LINEAR DRIFT}
\author{Anton Arnold}
\address{Institute for Analysis and Scientific Computing, TU Vienna, Wiedner Hauptstra\ss e 8-10, 1040 Vienna, Austria}
\email{anton.arnold@tuwien.ac.at}
\author{Christian Schmeiser}
\address{Faculty of Mathematics, University of Vienna, Oskar-Morgenstern-Platz 1, 1090 Vienna, Austria}
\email{Christian.Schmeiser@univie.ac.at}
\author{Beatrice Signorello}
\address{Institute for Analysis and Scientific Computing, TU Vienna, Wiedner Hauptstra\ss e 8-10, 1040 Vienna, Austria}
\email{bsignore@tuwien.ac.at}
\date{\today} 
\begin{document}

\maketitle
\begin{abstract}
We are concerned with the short- and large-time behavior of the $L^2$-propagator norm of Fokker-Planck equations with linear drift, i.e. $\partial_t f=\mathrm{div}_{x}{(D \nabla_x f+Cxf)}$. With a coordinate transformation these equations can be normalized such that the diffusion and drift matrices are linked as $D=C_S$, the symmetric part of $C$. The main result of this paper \bea{(Theorem \ref{maintheo})} is the connection between normalized Fokker-Planck equations and their drift-ODE $\dot x=-Cx$: Their $L^2$-propagator norms actually coincide. This implies that optimal decay estimates on the drift-ODE (w.r.t.\ both the maximum exponential decay rate and the minimum multiplicative constant) carry over to sharp exponential decay estimates of the Fokker-Planck solution towards the steady state.
A second application of the theorem regards the short time behaviour of the solution: The short time regularization (in some weighted Sobolev space) is determined by its hypocoercivity index, which has recently been introduced for Fokker-Planck equations and ODEs (see \cite{AE, AAC18, AAC}). 
\\ In the proof we realize that the evolution in each invariant spectral subspace can be represented as an explicitly given, tensored version of the corresponding drift-ODE. In fact, the Fokker-Planck equation can even be considered as the second quantization of $\dot x=-Cx$.
\end{abstract}

\section{Introduction}\label{intro}
We are going to study the large-time and short-time behavior of the solution of Fokker-Planck (FP) equations with linear drift  
and possibly degenerate diffusion for $g=g(t,y)$:
\begin{align}
\label{FP}
& \partial_{t} g = -\widetilde{L}g := \mathrm{div}_{y}{(\widetilde{D} \nabla_y g+\widetilde{C}yg)}, \qquad y \in \R^d, \ t \in (0,\infty),
\\
& g(t=0)=g_0 \in L_+^1(\R^d) \,,
\\
& \int_{\R^d} g_0(y) dy=1 \,.
\end{align}
We assume that
\begin{itemize}
\item $\widetilde{D}\in \R^{d \times d}$  is non-\bea{zero}, positive semi-definite, symmetric, and constant in $y$,
\item $\widetilde{C} \in \R^{d \times d} $ is positive stable, (typically non-symmetric,) and constant in $y$.
\end{itemize}

The goal of this study is to investigate the qualitative and quantitative large time behavior of the solution of \eqref{FP}. Several authors (see, e.g., \cite{AE}, \cite{AMTU}, \cite{PRT}, \cite{AEW}) have addressed the following questions: Under which conditions is there a non trivial steady state $g_{\infty}$? In the affirmative case, does the solution $g(t)$ converge to the steady state for $t \rightarrow \infty$ in a suitable norm? Is the convergence exponential?

In particular, the large-time behavior of FP-equations has been treated in \cite{Shi} via spectral methods. Instead, entropy methods are used in \cite{AMTU}.
From these previous studies it is well known that (under some assumptions that will be defined in the next section) the solution $g(t)$ converges to the steady state $g_{\infty}$ with an exponential decay rate, up to a multiplicative constant greater than one. In the degenerate case, where the diffusion matrix $\widetilde {D}$ is non-invertible, this property of the solution is known as \emph{hypocoercivity,} as introduced in \cite{Vi09}. 

Optimal exponential decay estimates for the convergence of the solution to the steady state in both the degenerate and 
the non-degenerate cases \bea{have} been shown in \cite{AE}.
Special care is required when the eigenvalues of $\widetilde{C}$ with smallest real part are \textit{defective}. This situation 
is covered in \cite{AEW} and \cite{M}. In both cases, the sharpness of the estimate refers only to the exponential decay rate of the convergence of the solution. The issue of finding the best multiplicative constant in the decay estimate for FP-equations \eqref{FP} is still open. This is one of the topics of this paper.
Even for linear ODEs there are only partial results on this best constant, as for example in \cite{MM} and \cite{AAS}.
In particular, \cite{AAS} gives the explicit best multiplicative constant in the two-dimensional case for 
$\dot x=-C x$, where $C$ is a positive stable matrix. A very complete solution has been derived in \cite{GM} for a special 
case, the kinetic FP-equation with quadratic confining potential. There the propagator norm is computed explicitly. The result 
can be written as an exponential decay estimate with time dependent multiplicative constant, whose maximal value is the 
result we are looking for. A related result based on Phi-entropies can be found in \cite{DL}, where improved time 
dependent decay rates are derived.

The main result of this paper \bea{(Theorem \ref{maintheo})} is equality of the propagator norms of the PDE on the orthogonal complement of the space
of equilibria and of its associated drift ODE. The underlying norms are the $L^2$-norm weighted by the inverse of the 
equilibrium distribution for the PDE, and the Euclidian norm for the ODE. 
This has two main consequences: First, the sharp (exponential) decay of the PDE is reduced to the same, but much easier question on the ODE level. The second consequence is that the hypocoercivity index (see \cite{AE, AAC18, AAC}) of the drift matrix determines the short-time behavior (in the sense of a Taylor series expansion) both of the drift ODE and the FP-equation. As a further consequence for solutions of the FP-equation we determine the short-time regularization from the weighted $L^2$-space to a weighted $H^1$-space. This result can be seen as an illustration of the fact that for the FP-equation hypocoercivity is equivalent to hypoellipticity.
Finally, it is shown that the FP-equation can be considered as the second quantization of the drift ODE. This follows from the
proof of the main theorem, where the FP-evolution is decomposed on invariant subspaces, in each of which the evolution is 
governed by a tensorized version of the drift ODE.

The paper is organized as follows: In Section 2 we transform the FP-operator $\tilde{L}$ to an equivalent version $L$ such that $D=C_S$, the symmetric part of the drift matrix. The conditions for the existence of a unique positive steady state 
and for hypocoercivity are also set up. The main theorem is formulated in Section 3 together with the main consequences.
The proof of the main theorem requires a long preparation that is split into Sections 4 and 5. In Section 4 we derive a spectral decomposition for the FP-operator into finite-dimensional invariant subspaces. This allows to see an explicit link with the 
drift ODE $\dot{x}=-C x$. In order to make this link more evident, we work with the space of symmetric tensors, presented 
in Section 5. In Section 6 we give the proof of the main theorem as a corollary of the fact that the propagator norm on each subspace is an integer power of the propagator norm of the ODE evolution.
Finally, in Section 7 the FP-operator is rewritten in the second quantization formalism.


\section{\bea{Preliminaries and main result}}\label{sec2}
\label{sec:second}

\subsection{Equilibria -- normalized \bea{Fokker-Planck equation}}
The following theorem (from \cite{AE}, Theorem 3.1 or \cite{MPP02}, p.\ 41) states under which conditions on the matrices $\widetilde{D}$ and $\widetilde{C}$ there exists a unique steady state $g_{\infty}$ for (\ref{FP}) and it provides its explicit form. We denote the
\textit{spectral gap of $\widetilde{C}$} by $\mu(\widetilde{C}): = \min \{ \Re (\lambda): \text{ $\lambda$ is an eigenvalue of $\widetilde{C}$} \}$.
\begin{definition}
\label{conAtilde}
We say that \textit{Condition $\widetilde{A}$} holds for the Equation \eqref{FP}, iff 
\begin{enumerate}
\item the matrix $\widetilde{D}$ is symmetric, positive semi-definite,
\item there is no non-trivial $\widetilde{C}^T$-invariant subspace of $\mathrm{ ker}{\widetilde{D}}$,
\item the matrix $\widetilde{C}$ is positive stable, i.e. $\mu(\widetilde{C})>0$.
\end{enumerate}
\end{definition}

Note that condition (2) is known as Kawashima's degeneracy condition \cite{Kawashima} in the theory for systems
of hyperbolic conservation laws. It also appears in \cite{Hoermander} as a condition for hypoellipticity of FP-equations 
(see \cite[Section 3.3]{Vi09} for the connection to hypocoercivity).

\begin{theorem}[Steady state]
\label{steadystate}
There exist a unique ($L^1$-normalized) steady state $g_{\infty}\in L^1(\R^d)$ of (\ref{FP}), iff
\textit{Condition $\widetilde{A}$} holds. It is given by the (non-isotropic) Gaussian
\begin{equation}\label{nonisogau}
   g_{\infty}(y)=c_{K}\exp{\left(-\frac{y^T K^{-1}y}{2} \right )} \,,
\end{equation}
where the covariance matrix $K\in \R^{d \times d}$ is the unique, symmetric, and positive definite solution of the continuous Lyapunov equation
\begin{equation}
\label{Lyapeq}
2\widetilde{D}=\widetilde{C}K+K\widetilde{C}^T,
\end{equation}
and $c_{K}=(2\pi)^{-d/2}(\det K)^{-1/2}$ is the normalization constant.
\end{theorem}

\bea{In the above theorem, the matrix $K$ can be represented analytically as
$$
  K=2 \int_0^\infty e^{-\widetilde{C}\tau}\widetilde{D} e^{-\widetilde{C}^T\tau} d\,\tau
$$
(see \cite{MPP02}, p.\ 41), and the numerical solution of \eqref{Lyapeq} can be obtained with the Matlab routine \emph{lyap}.}

Under Condition $\widetilde{A}$ the FP-equation \eqref{FP} can be rewritten (see Theorem $3.5$, \cite{AE}) as
\begin{equation}
\label{compactFP1}
\partial_{t} g=   \mathrm{div}_{y}{\left(g_{\infty} (\widetilde{D}+\widetilde{R}) \nabla_y{\left( \frac{g}{g_{\infty}}\right)} \right) }, 
   \qquad y \in \R^d, \ t \in (0,\infty),
\end{equation}
where $\widetilde{R} \in \mathbb \R^{d \times d }$ is the anti-symmetric matrix $\widetilde{R}=\frac{1}{2}\left (\widetilde{C}K-K\widetilde{C}^T \right )$.
The natural setting for the evolution equation \eqref{FP} is the weighted $L^2$-space $\mathcal{\widetilde{H}}:=L^2(\R^d, g_{\infty}^{-1})$ with the inner product 
$$\langle g_1,g_2 \rangle _{\mathcal{\widetilde{H}}}:=\int_{\R^d}g_1(y)g_2(y) \frac{dy}{g_{\infty}(y)} \,.$$

\bea{Using the notations $\widetilde V_0:=\operatorname{span}_\R\{g_\infty\}\subset \mathcal{\widetilde{H}}$ and $C:=K^{-1/2}\widetilde{C}K^{1/2}$ we can now formulate the main result of this paper:
\footnote{\bea{Note added in print: In the follow-up paper \cite{ASi}, Theorem \ref{maintheo1} was recently extended to FP-equations with time dependent coefficient matrices $\tilde D(t)$, $\tilde C(t)$, provided that all these FP-operators with fixed $t$ have the same steady state, i.e.\ if \eqref{Lyapeq} holds for all $t$ with a constant matrix $K$. In this extension the two propagators in \eqref{mainequality1} are replaced by the propagation operators that map the solution at time $t_1$ to the solution at time $t_2 \ge t_1$, both for the FP-equation and for the corresponding drift ODE  $\frac{d}{dt}x=-C(t)x$. $K$ being constant in time implies that the FP-normatization to \eqref{NormalizedFPwithC}, the spaces $\mathcal{H}$ and $\mathcal{\widetilde{H}}$, as well as the subspace decomposition in \S\ref{sec:4.1} are all time independent.}}
\begin{theorem}
\label{maintheo1}
Let Condition $\widetilde A$ hold for the FP-equation \eqref{FP}.
Then the propagator norms of the FP-equation \eqref{FP} and its corresponding drift ODE $\frac{d}{dt}x=-Cx$ are equal,
i.e.,
\begin{equation}
\label{mainequality1}
    \left\|e^{-\widetilde L t}\right\|_{\mathcal{B}(\widetilde V_0^{\perp})} = \left\|e^{-Ct}\right\|_{\mathcal{B}(\R^d)} \,, \qquad\forall t \geq 0 \,,
\end{equation}
where $\mathcal{B}(.)$ denotes the operator and spectral matrix norms (for more details see Definition \ref{def:propNorm} below).
\end{theorem}
The fact that \eqref{mainequality1} involves the matrix $C$ (and not $\widetilde C$), motivates to introduce the following coordinate transformation.
Using }
$x:=K^{-1/2}y$, $f(x):=(\det K)^{1/2}g(K^{1/2}x)$ transforms \eqref{FP} into
\begin{equation}
\label{NormalizedFPwithC}
\partial_{t} f=  -Lf:= \mathrm{div}_{x}{(D \nabla_x f + Cxf)} = \mathrm{div}_{x}{\left(f_{\infty}C\nabla_x{\left( \frac{f}{f_{\infty}}\right)} \right) } \,,  
\end{equation}
where $D:=K^{-1/2}\widetilde{D}K^{-1/2}$, 
and the steady state is the normalized Gaussian 
\begin{equation}\label{f-infty}
  f_\infty(x) = (2\pi)^{-d/2}e^{-|x|^2/2} \,. 
\end{equation}
This is due to the property
\begin{equation}\label{norm-prop}
    D = C_S := \frac{1}{2}\left(C+C^T\right) \,,
\end{equation}
which is a simple consequence of \eqref{Lyapeq}. We shall call a FP-equation \emph{normalized,} if the diffusion and drift 
matrices satisfy \eqref{norm-prop}.

For later reference we rewrite Condition $\tilde{A}$ in terms of the matrix $C$: 
\bea{
\begin{definition}
\label{conA}
We say that \textit{Condition $A$} holds for the Equation \eqref{NormalizedFPwithC}, iff 
\begin{enumerate}
\item the matrix $C_S$ is positive semi-definite,
\item there is no non-trivial $C^T$-invariant subspace of  \ $\mathrm{ker}{C_S}$.
\end{enumerate}
\end{definition}
}

\bea{
\begin{prop}
\label{Prop2.5}
The Equation \eqref{FP} satisfies Condition $\widetilde{A}$ \textit{iff} its normalized version \eqref{NormalizedFPwithC}
satisfies Condition $A$.
Moreover, Condition A implies that the matrix $C$ is positive stable, i.e. $\mu(C)>0$.
\end{prop}
}

\begin{proof}
Equivalence of \bea{the items ($1$) in Definitions \ref{conAtilde} and \ref{conA}} follows from $C_S=K^{-\frac{1}{2}}\widetilde{D} K^{-\frac{1}{2}}$.
For the second item, let us assume that $(2)$ \bea{in Definition \ref{conA}} does not hold. Then, there exist $v \in \mathrm{ker}{C_S}, v \neq 0 \in \R^d$ such that 
$$
0=C_S C^Tv=(K^{-1/2}\widetilde{D}K^{-1/2})(K^{1/2}\widetilde{C}^TK^{-1/2})v=K^{-1/2}\widetilde{D} \widetilde{C^T} (K^{-1/2}v).
$$
This implies $\widetilde{D} \widetilde{C}^T(K^{-1/2}v)=0$, since $K^{-1/2}>0$. 
But this is a contradiction to $(2)$ in Condition $\widetilde{A}$ since it holds that
$v \in \mathrm{ker}{C_S} \text{ iff } K^{-1/2}v \in \mathrm{ker}{\widetilde{D}}$. 
With a similar argument the reverse implication can be proven.

For the proof that Condition A implies positive stability of $C$ we refer to Proposition $1$ and Lemma $2.4$ in \cite{AAC18}.
\end{proof}

From now on we shall study the normalized equation \eqref{NormalizedFPwithC} on the normalized version 
$\mathcal{H}:=L^2\left(\mathbb{R}^d,f_{\infty}^{-1}\right)$ of the Hilbert space $\mathcal{\widetilde{H}}$. It is easily \linebreak checked 
that
	\begin{equation}
	\label{equivEQ}
	\|g(t)\|_{\widetilde{\mathcal{H}}}=\|f(t)\|_{\mathcal{H}}, \qquad \forall t \geq 0,
	\end{equation}
holds for the solutions $g$ and $f$ of \eqref{FP} and, respectively, \eqref{NormalizedFPwithC}. \bea{This implies that the propagator norms for $\widetilde L$ and $L$ are the same, and that the Theorems \ref{maintheo1} and \ref{maintheo} are equivalent.}

\subsection{Convergence to the equilibrium: hypocoercivity}
In \cite{AE}, a hypocoercive entropy method was developed to prove the exponential convergence to $f_{\infty}$, for the solution to \eqref{NormalizedFPwithC} with any initial datum $f_0 \in \mathcal{H}$.
It employed a family of relative entropies w.r.t.\ the steady state, i.e. $e_{\psi}(f(t)|f_{\infty})$ $:= \int _{\R^d} \psi \left ( \frac{f(t)}{f_{\infty}} \right ) \bea{f_{\infty}} dx$, where the convex functions $\psi$ are admissible entropy generators (as in \cite{AMTU} and \cite{BE}). 
\begin{definition}
\label{defdefective}
\bea{Given} $\mu(C):=\min\{ \mathrm{Re}{(\lambda)}: \lambda \text{ is an eigenvalue of $C$}\}$.
\begin{enumerate}
\item We call the matrix $C$ \textit{non-defective} if \bea{all the eigenvalues $\lambda$ with $\mathrm{Re}{(\lambda)}=\mu(C)$} are \textit{non-defective,} i.e.,
their algebraic and geometric multiplicities coincide.
\item We call a FP-equation \eqref{FP} (non-)defective if its drift-matrix $\widetilde{C}$ is 
(non-\-)defective, or equivalently, if the matrix $C$ in the normalized version \eqref{NormalizedFPwithC} is (non-)defective.
\end{enumerate}
\end{definition}

For non-defective FP-equations, the decay result from \cite{AE} provides \bea{on the one hand} the sharp exponential decay rate $\mu>0$, but, \bea{on the other hand, only a }sub-optimal multiplicative constant $c>1$. \bea{We give a slightly modified version of it:}
\begin{theorem}[Exponential decay of the relative entropy, \bea{Theorem 4.9, \cite{AE}}]
\label{theoFP}
Let $\psi$ generate an admissible entropy and let $f$ be the solution of (\ref{NormalizedFPwithC}) with normalized initial state $f_0\in L^1_+(\R^d)$ such that $e_{\psi}(f_0|f_{\infty})<\infty$. Let $C$ satisfy Condition $A$.
Then, if the FP-equation is non-defective, there exists a constant $c\geq 1$ such that
\begin{equation}
\label{entropydecay}
e_{\psi}(f(t)|f_{\infty}) \leq c^2e^{-2\mu t} e_{\psi} (f_0|f_{\infty}), \quad t \geq 0.
\end{equation}
\end{theorem}

Choosing the admissible quadratic function $\psi(\sigma)=(\sigma-1)^2$ yields the exponential decay of the $\mathcal{H}$-norm. For this particular choice of $\psi$, Theorem \ref{theoFP} holds also for $f_0 \in L^1(\R^d) \cap \mathcal{H}$, i.e. the positivity of the initial datum $f_0$ is not necessary.
\begin{corollary}[Hypocoercivity]
Under the assumptions of Theorem \ref{theoFP} the following estimate holds with the same $\mu>0$, \bea{$c\geq1$}: 
\begin{equation}
\label{hypo}
\|f(t)-f_{\infty}\|_{\mathcal{H}} \leq c e^{- \mu t } \|f_0-f_{\infty}\|_{\mathcal{H}}, \quad t \geq 0.
\end{equation}
\end{corollary}
The hypocoercivity approach in \cite{AE} provides the optimal (i.e. maximal) value for $\mu$ and a computable value for
$c$, which is however not sharp, i.e. $c>c_{\min}$ with
\begin{equation}
\label{defcmin}
c_{\min}:= \min\left\{ c \geq 1: \ \eqref{hypo} \text{ holds for all } f_0 \in \mathcal{H} \text{ with } \int_{\R^d}{ f_0\, dx}=1 \right\}.
\end{equation}
\bea{One} central goal of this paper is the determination of $c_{\min}$. \bea{But,} actually, we shall go much beyond this:
The main result of this paper, \bea{Theorem \ref{maintheo}, states }that the $\mathcal{H}$-propagator norm of each (stable) FP-equation is \textit{equal} to the \bea{(spectral)} propagator norm of its corresponding drift ODE $\dot{x}(t)=-Cx(t)$. Hence, all decay properties of the FP-equation \eqref{FP} can be obtained from a simple linear ODE, and sharp exponential decay estimates of \bea{this} ODE carry over to the corresponding FP-equation.
\bea{So, for quantifying the decay behavior of FP-equations with linear drift, an infinite dimensional PDE problem can be replaced by a (small) finite dimensional ODE problem.}

\subsection{The best multiplicative constant for \bea{the ODE-decay}}\label{bestODE}
In \cite{AAS} we analyzed the best decay constants for the (of course easier) finite dimensional problem
\begin{equation}
\label{eqODE}
\dot x(t)=-C x(t) \,, \quad t> 0 \,,\qquad x(0)=x_0\in \mathbb{C}^n \,, 
\end{equation}
where $C \in \mathbb \C^{n\times n}$ is a positive stable and non-defective matrix. 
In this case we constructed a problem adapted norm as a Lyapunov functional. This allowed to derive a hypocoercive estimate for the Euclidean norm $\| \cdot\|_2$ of the solution:
\begin{equation}
\label{hypoODE}
\|x(t)\|_2 \leq c e^{-\mu t} \|x_0\|_2, \qquad  t \geq 0 \,.
\end{equation}
Here $\mu>0$ is the spectral gap of the matrix $C$ (and the sharp decay rate of the ODE \eqref{eqODE}), and $c\geq 1$ is some constant.

In \cite{AAS} we investigated, in the two dimensional case, the sharpness of the constant $c$. 
By analogy with \eqref{defcmin}, we define the best multiplicative constant for the hypocoercivity estimate of the ODE as 
\begin{equation*}
c_1:=c_1(C):= \min\left\{ c \geq 1: \ \eqref{hypoODE} \text{ holds for all } x_0 \in \C^n\right\} \,.
\end{equation*}
The explicit expression for the best constant $c_1$ depends on the spectrum of $C$. \bea{In \cite{AAS} we treated all the cases for matrices in $\C ^{2\times 2}$}. In particular, denoting by $\lambda_1, \lambda_2$ the two eigenvalues of $C$, we distinguish three cases:
\begin{enumerate}
\item $\Re (\lambda_1)=\Re (\lambda_2)=\mu$;
\item $\mu= \Re (\lambda_1) < \Re( \lambda_2)$, $\Im (\lambda_1)=\Im (\lambda_2)$;
\item $\mu = \Re (\lambda_1) < \Re( \lambda_2)$, $\Im (\lambda_1) \neq \Im (\lambda_2)$.
\end{enumerate}
The corresponding explicit form of $c_1$ in the cases $(1)$ and $(2)$ is described in the next theorem (\bea{see Theorem 3.7 and Theorem 4.1 in \cite{AAS}}). For the case $(3)$ we have, instead, an implicit form, see \bea{Proposition 4.2 and Corollary 4.3 in \cite{AAS}}.

\begin{theorem}
\label{theoAAS}
Let $C\in \mathbb{C}^{2 \times 2} $ be positive stable and non-defective with eigenvalues $\lambda_1,\lambda_2$. Denoting 
by $\alpha \in [0,1)$ the cosine of the angle between the two eigenvectors of $C^T$, the best constant  for $\eqref{hypoODE}$ in the cases $(1)$ and $(2)$ is
$$c_1=\sqrt{\frac{1+\alpha}{1-\alpha}} \qquad\mbox{and, respectively,}\qquad c_1=\frac{1}{\sqrt{1-\alpha^2}} \,.$$
\end{theorem}
For dimension $n \geq 3$, explicit expressions for the best constant $c_1$ seem to be unknown in general.
\\
\begin{subsubsection}{The defective case}
So far we have discussed non-defective matrices \bea{$C \in \R^{d \times d}$}. The remaining case has to be treated apart since we cannot obtain both the optimality of the multiplicative constant and the sharpness of the exponential decay at the same time if $C$ is defective. Nevertheless, hypocoercive estimates \bea{do} hold (see Chapter 1.8 in \cite{P} and Theorem 2.8 in \cite{AJW}) with either reduced exponential decay rates \bea{(see Theorem $4.9$ in \cite{AE})} or with the best decay rate $\mu$, but augmented with a time-polynomial coefficient \bea{(see Theorem $2.8$ in \cite{AJW})}, as the following theorem claims.
\begin{theorem}
\label{hypodefectiveODE}
Let \bea{$C \in \C^{d \times d}$} be a positive stable (possibly defective) matrix with spectral gap $\mu >0$. Let $M$ be the maximal size of a Jordan block associated to $\mu$. Let $x(t)$ be the solution of the ODE $\frac{d}{dt}x(t)=-Cx(t)$ with initial datum $x_0 \in \C^d$. Then, for each $\epsilon >0$  there exist a constant $c_{\epsilon}\geq 1$ such that 
\begin{equation}
\label{1hypoODEdefect}
\|x(t)\|_2 \leq c_{\epsilon} e^{-(\mu-\epsilon)t} \|x_0\|_2, \qquad \forall t \geq 0, x_0 \in \C^d. 
\end{equation}
Moreover, there exists a polynomial $p(t)$ of degree $M-1$ such that 
 \begin{equation}
\label{2hypoODEdefect}
\|x(t)\|_2 \leq p(t) e^{-\mu t} \|x_0\|_2, \qquad \forall t \geq 0, x_0 \in \C^d.
\end{equation}
\end{theorem}
As we did for the non-defective case, we define the best constant $c_{1,\epsilon}$ for the estimate \eqref{1hypoODEdefect} with rate $\mu-\epsilon$ as
\begin{equation*}
c_{1,\epsilon}:= \min \left\{c_\epsilon \geq 1: \ \eqref{1hypoODEdefect} \ \text{holds for all } x_0 \in \C^d \right\} \,.
\end{equation*}
We do not attempt to define an "optimal polynomial" $p(t)$ in \eqref{2hypoODEdefect}.
In the next section it is shown that these ODE-results carry over to the corresponding FP-equation \eqref{NormalizedFPwithC}.
\end{subsubsection}

\begin{section}{Main result \bea{for normalized FP-equations} and applications}
\bea{In Theorem \ref{maintheo1} we anticipated the main result of this paper for the \emph{non-normalized} FP-equation \eqref{FP}. In the sequel we shall deal with its equivalent formulation for \emph{normalized} FP-equations, since this will simplify the proof.}
With the above review of ODE results we can now state \bea{an essential aspect of this main result}: The best decay constants in \eqref{hypo} for the FP-equation \eqref{NormalizedFPwithC} (and therefore also for \eqref{FP}) coincide with the best constants for the ODE \eqref{eqODE}.
This result is a corollary of the main theorem of this paper, \bea{namely Theorem \ref{maintheo}}. 
It claims that the propagator norm of the FP-equation coincides with the propagator norm of its corresponding ODE (w.r.t. the Euclidean \bea{vector norm). With \emph{propagator norm} we refer to the following notion for linear ODEs or PDEs: If $A$ is their infinitesimal generator on some Banach space $X$ and $e^{At},\,t\ge0$ their \emph{propagator}, forming a $\mathcal{C}_0$-semigroup of bounded operators (cf.\ \cite{Pa83}), the propagator norm is the operator norm of $e^{At}$ on $X$, see Definition \ref{def:propNorm} below.}

First we define the projection operator $\Pi_0$ that maps a function in $\mathcal{H}$ into the subspace generated by the steady state $f_{\infty}$.
\begin{definition}
\label{proj0}
Let \bea{$f\in \mathcal{H} = L^2\left(\mathbb{R}^d,f_{\infty}^{-1}\right)$} and $f_\infty$ the normalized Gaussian \eqref{f-infty}. We define the operator $\Pi_0: \mathcal{H}\longrightarrow \mathcal{H}$ as 
$$\Pi_0f:=\langle f,f_\infty \rangle_{\mathcal{H}}f_{\infty},$$
i.e., $\Pi_0$ projects $f$ onto $V_0:=\mathrm{span}_{\R}{\{f_{\infty}\}} = \mathcal{N}(L)$.
\end{definition}
\begin{remark}
Let $f \in \mathcal{H}$. Then, the coefficient $\langle f,f_{\infty} \rangle_{\mathcal{H}}$ is equal to $\int_{\R^d} f(x) dx$, by definition. Moreover, it is obvious \bea{from the divergence form} of \eqref{NormalizedFPwithC} that the "total mass"  $\int_{\R^d}f(t,x) dx$ remains constant in time under the flow of the equation. Hence, $(\Pi_0f)(t)$ is independent of $t$, if $f(t)$ solves \eqref{NormalizedFPwithC}. This implies $e^{-Lt}(\mathbb{1} - \Pi_0)=e^{-Lt}- \Pi_0$.
\end{remark}

We introduce the standard definitions of operator norms.
\begin{definition}
\label{def:propNorm}
Let $A:\,\mathcal{H}\to\mathcal{H}$ and $B:\, \R^d\to\R^d$ be linear operators. Then 
$$
    \|A \|_{\mathcal{B}(\mathcal{H})} := \sup_{0 \neq f \in \mathcal{H}} \frac{\|Af\|_{\mathcal{H}}}{\|f\|_{\mathcal{H}}} \,,\qquad
        \|B\|_{\mathcal{B}(\R^d)} := \sup_{0 \neq x \in \R^d} \frac{\|Bx\|_2}{\|x\|_2} \,.
$$
\end{definition}
\bigskip

If $f(t)$ is the solution of the FP-equation \eqref{NormalizedFPwithC}  with $f(0)=f_0 \in \mathcal{H}$, then
$$
  \left\|e^{-Lt} \left(  \mathbb{1} - \Pi_0 \right) \right\|_{\mathcal{B}(\mathcal{H})} = \left\|e^{-Lt} \right\|_{\mathcal{B}(V_0^\bot)}
   = \sup_{0 \neq f_0 \in \mathcal{H}}{\frac{\|f(t)-\Pi_0f_0\|_{\mathcal{H}}}{\|f_0\|_{\mathcal{H}}}} \,.
$$
If $x(t) \in \R^d$ is the solution of the ODE $\frac{d}{dt}x=-Cx\,$ with initial datum $x(0):=x_0$, then
\begin{equation*}
   \left\|e^{-Ct}\right\|_{\mathcal{B}(\R^d)} = \sup_{0 \neq x_0 \in \R^d}{\frac{\|x(t)\|_2}{\|x_0\|_2}}.
\end{equation*}
With these notations we can state the main result of this paper.
\begin{theorem}
\label{maintheo}
Let Condition $A$ hold for the FP-equation \eqref{NormalizedFPwithC}.
Then the propagator norms of the FP-equation \eqref{NormalizedFPwithC} and its corresponding ODE $\frac{d}{dt}x=-Cx$ are equal,
i.e.,
\begin{equation}
\label{mainequality}
    \left\|e^{-Lt}\right\|_{\mathcal{B}(V_0^{\perp})} = \left\|e^{-Ct}\right\|_{\mathcal{B}(\R^d)} \,, \qquad\forall t \geq 0 \,.
\end{equation}
\end{theorem}
The proof of Theorem \ref{maintheo} will be prepared in the following two sections and finally completed in Section \ref{sec:fifth}.

Theorem \ref{maintheo} can be seen as a generalization of a result in \cite{GM}, where the propagator norm for the following kinetic FP-equation \bea{(the $L^2$-adjoint equation of (2) in \cite{GM})}
\begin{eqnarray}
\partial_t g &=& -\widetilde L_a g := -v\, \partial_x g + \partial_v(\partial_v g +  (ax+v)g) \nonumber\\
  &=& \mathrm{div}_{(x,v)} \left( \begin{pmatrix} 0 & 0 \\ 0 & 1\end{pmatrix} \nabla_{(x,v)} g 
  + \begin{pmatrix} 0 & -1 \\ a & 1\end{pmatrix}\begin{pmatrix}x\\ v\end{pmatrix}g\right)\,,\label{kinFP}
\end{eqnarray}
with $(x,v)\in \R^2$ and the parameter $a>0$, has been computed explicitly.

\begin{theorem}
\label{theoGM} \cite[Theorem 1.2]{GM}
For any $a>0$ and $t \geq 0$, it holds: 
\begin{equation}
\label{NormPta}
\left\|e^{-\widetilde{L}_a t}\right\|_{\mathcal{B}(V_0^{\perp})}=c_{a}(t) \exp{\left (    - \frac{1-\sqrt{(1-4a)_+}}{2} t      \right )},
\end{equation}
where the non-negative factor $c_a(t)$ is given for $0<a<1/4$ by
\begin{equation}
\label{GMcase1}
c_a(t):=\sqrt{ e^{-2\theta t} + \frac{1-\theta^2}{2 \theta ^2}  ( 1- e^{-\theta t} )^2 + \frac{1-e^{-2\theta t}}{2} \left (   1+\frac{1}{\theta}  \sqrt{1+(\theta^{-2}-1) \left (  \frac{e^{\theta t}-1}{e^{\theta t}+1} \right )^2}    \right )} \,,
\end{equation}
with $\theta = \sqrt{1-4a}$, for $a>1/4$ by
\begin{equation}\label{GMcase2}
c_a(t):= \sqrt{1+\frac{|e^{\theta t }-1|}{2|\theta|^2}  \left ( |e^{\theta t }-1| + \sqrt{|e^{\theta t }-1|^2 +4|\theta|^2}   \right )  } \,,
\end{equation}
with $\theta:= \sqrt{ 4 a -1}i$, and for $a=1/4$ by
\begin{equation}
\label{eq:Cadefective}
c_a(t):= \sqrt{1+ \frac{t^2}{2}+ t \sqrt{1+\left(\frac{t}{2} \right )^2}} \,\,.
\end{equation}
\end{theorem}
Note that there is a small typo in the formula for $c_a(t)$, $a<1/4$ in \cite{GM} that corresponds to \eqref{GMcase1}.

\bea{After normalization of the FP-equation \eqref{kinFP}, the corresponding drift matrix is given by}
\begin{equation}
\label{Ca}
C_{a}:= \left ( \begin{matrix}
0 & -\sqrt{a} \\ \sqrt{a} & 1
\end{matrix}
\right ).
\end{equation}
Its eigenvalues are $\lambda_{1,2}:=\frac{1}{2} \left(1 \pm \theta \right)$,
with $\theta$ as in Theorem \ref{theoGM}, and the corresponding eigenvectors are $v_{1,2} = (\sqrt{a}, -\lambda_{1,2})^T$. 
This shows that the spectral gap is given by 
$\mu = \frac{1}{2}\left(1-\sqrt{(1-4a)_+}\right)$. It is easy to check that $C_{a}$ satisfies Condition $A$ for each $a>0$. We observe that the value $a=1/4$ is critical in the sense that $C_{1/4}$ is defective. 

With the approach of this work we can employ the results of Section \ref{bestODE} for obtaining the best possible constant
$c_1$ in
$$
   \left\|e^{-\widetilde{L}_a t}\right\|_{\mathcal{B}(V_0^{\perp})} = \left\|e^{-C_a t}\right\|_{\mathcal{B}(\R^d)} \le c_1 e^{-\mu t} \,.
$$
For $a\ne 1/4$ we apply Theorem \ref{theoAAS} and note that for $0<a<1/4$ we are in case (2). We compute $\alpha = 2\sqrt{a}$, giving the optimal constant
$$
   c_1 = (1-4a)^{-1/2} \,,
$$
which can also be obtained from \eqref{GMcase1} in the limit $t\to\infty$. For $a>1/4$ we are in case (1) and obtain
$\alpha = (2\sqrt{a})^{-1}$ and 
$$
  c_1 = \frac{2\sqrt{a} + 1}{\sqrt{4a-1}} \,.
$$
The same is obtained as the maximal value of $c_a(t)$ in \eqref{GMcase2}, taken whenever $\left|e^{\theta t}-1\right|=2$.

Finally, for $a=1/4$ the results of Theorems \ref{hypodefectiveODE} and \ref{theoGM} agree with $c_a(t)\approx t$ as
$t\to\infty$, \bea{since the best approximation for the function in \eqref{eq:Cadefective}, i.e.\ the smallest affine linear upper bound to \eqref{eq:Cadefective}, is the polynomial $p(t)=1+t$.}

The plot in Figure \ref{plot1} shows the right-hand side of \eqref{NormPta} as a function of time for 3 values of $a$ ($a=1/5$, $a=1/4$, $a=2$). Note the non-smooth behavior in the case $a=2$.
\begin{figure}[H]
\includegraphics[scale=0.6]{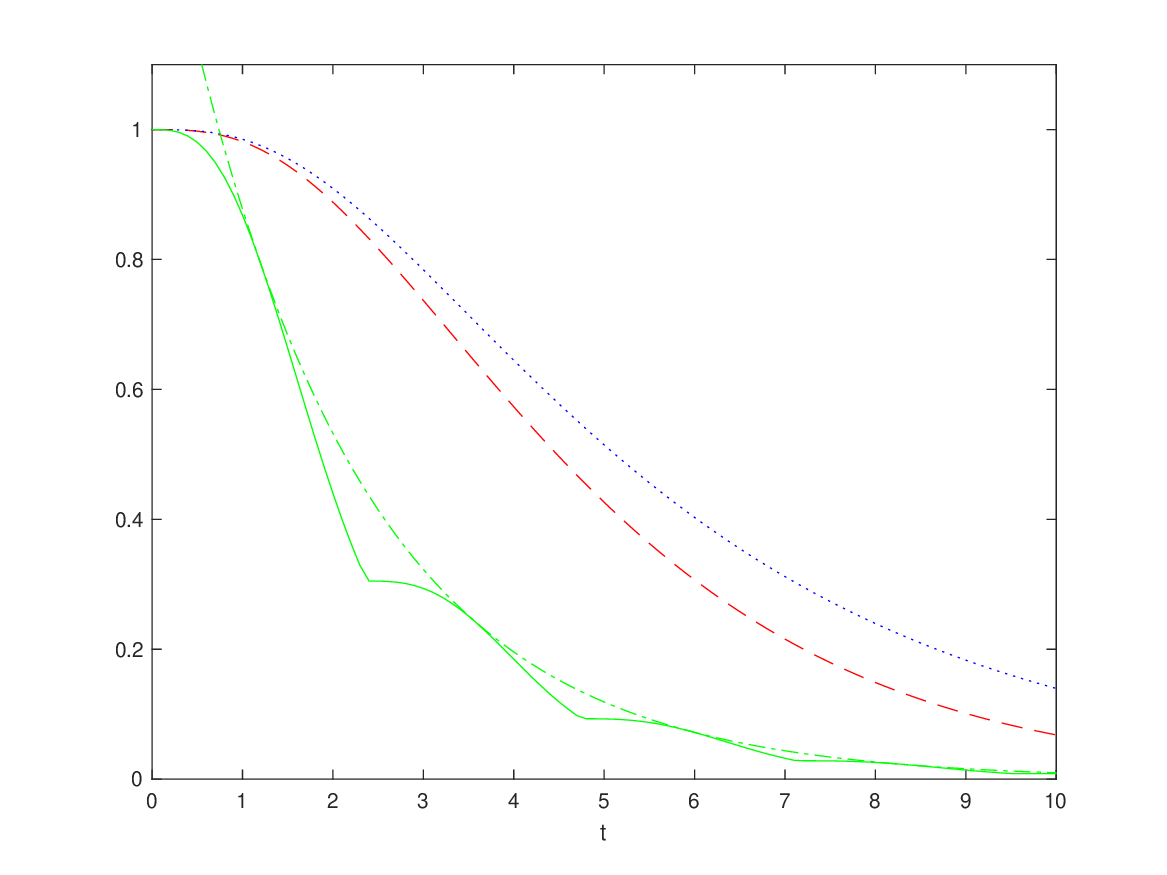}
\caption{The propagator norm for equation \eqref{kinFP} for 3 values of the parameter $a$. Solid (green) curve for $a=2$, dashed (red) curve for $a=1/4$, dotted (blue) curve for $a=1/5$. The dash-dotted (green) curve, gives the best exponential bound of the form $c_1 e^{-t/2}$ for the case $a=2$. \bea{Note: The curves are colored only in the electronic version of this article.} }
\label{plot1}
\end{figure}

\begin{subsection}{Applications of Theorem \ref{maintheo}}
\subsubsection{Long time behavior}
One consequence of Theorem $\ref{maintheo}$ is that all the estimates about the decay of the solutions of the ODE carry over to the corresponding FP-equation. In particular, it follows that the hypocoercive ODE estimates \eqref{hypoODE} and \eqref{1hypoODEdefect} hold also for solutions of the corresponding FP-equation. Moreover, the best constants in the estimates are the same both for the FP-case and for its corresponding drift ODE. 
\begin{theorem}
\label{ilteorema}
Let $C \in \R^{d \times d}$ be \textit{non-defective} and satisfy Condition $A$. Let $c_1$ be the best constant in the estimate \eqref{hypoODE} for the ODE \eqref{eqODE}. Then it is also the optimal constant $c_{\min}$ in the following hypocoercive estimate 
\begin{equation}
\label{firstestimate}
\|f(t)-f_{\infty}\|_\mathcal{H} \leq c_1 e^{-\mu t} \|f_0-f_{\infty}\|_{\mathcal{H}}, \quad \forall t\geq 0, \forall f_0 \in \mathcal{H}, \int_{\R^d}f_0(x)\, dx=1
\end{equation}
for the \bea{solution of the} FP-equation \eqref{NormalizedFPwithC}.
\end{theorem}
\begin{theorem}
\label{ilteoremaDEFECTIVE}
Let $C\in \R^{d \times d}$ be defective and satisfy Condition $A$. Let $M$ be the maximal size of a Jordan block associated to $\mu$. Let $\epsilon>0$ be fixed 
and $c_{1, \epsilon}$ be the best constant in the estimate \eqref{1hypoODEdefect} for the ODE \eqref{eqODE}. Then the following hypocoercive estimate holds
\begin{equation}
\label{firstestimateDEF}
\|f(t)-f_{\infty}\|_\mathcal{H} \leq c_{1,\epsilon}e^{-(\mu - \epsilon)t} \|f_0-f_{\infty}\|_{\mathcal{H}}, \quad \forall t\geq 0, \forall f_0 \in \mathcal{H}, \int_{\R^d}f_0(x)\, dx=1
\end{equation}
for the \bea{solution of the} FP-equation \eqref{NormalizedFPwithC}, and $c_{1, \epsilon}$ \bea{is the optimal multiplicative constant}.
Moreover, 
\begin{equation}
\label{secondtestimateDEF}
\|f(t)-f_{\infty}\|_\mathcal{H} \leq p(t) e^{-\mu t} \|f_0-f_{\infty}\|_{\mathcal{H}}, \quad \forall t\geq 0, \forall f_0 \in \mathcal{H}, \int_{\R^d}f_0(x)\, dx=1,
\end{equation}
where $p(t)$ is the polynomial of degree $M-1$ appearing in \eqref{2hypoODEdefect}.
\end{theorem}
We \bea{remind} that the quest to obtain the best decay for \eqref{FP} is \bea{thus} reduced to the knowledge of the best decay constants for the corresponding drift ODE.
\subsubsection{Short time behavior}
The second application of Theorem \ref{maintheo} concerns the short time behavior of the propagator norm of the FP-operator. It is linked to the concept of \textit{hypocoercivity index}, which describes the "structural complexity" of the matrix $C$ and, more precisely, the intertwining of its symmetric and anti-symmetric parts. For the FP-equation, the hypocoercivity index reflects its degeneracy structure. As we are going to illustrate in this section, this index represents the polynomial degree in the short time behavior of the propagator norm, both in the FP-equation and in the ODE case. Moreover it describes the rate of regularization of the FP-solution from $\mathcal{H}$ to a weighted Sobolev space $H^1$.

\bea{Next we recall the definition of \textit{hypocoercivity index} both for 
FP-equa\-tions and ODEs, respectively, from \cite{AE} and \cite{AAC18, AAC}}. We will see that these two concepts coincide when we consider the drift ODE associated to the FP-equation.
We first give the definition for the normalized FP-equation and then it will be illustrated that the index is invariant for the general ($D \neq C_S$) equation \eqref{FP}.
\begin{definition}
\label{defHCI1}
We define $m_{HC}$, the \textit{hypocoercivity index} for the normalized FP-equation \eqref{NormalizedFPwithC} as the minimum $m \in \mathbb{N}_0$ such that
\begin{equation}
T_m:= \sum_{j=0}^m C_{AS}^j C_S (C_{AS}^T)^j >0 \,.
\end{equation}
Here $C_{AS}:=\frac{1}{2}(C-C^T)$ denotes the anti-symmetric part of $C$.
\end{definition}
\begin{remark}
Lemma 2.3 in \cite{AE} states that the condition $m_{HC}< \infty$ is equivalent to the FP-equation being hypoelliptic. This index can be seen as a measure of "how much'' the drift matrix has to mix the directions of the kernel of the diffusion matrix with its orthogonal space in order to guarantee convergence to the steady state.
For example, $m_{HC}=0$ means, by definition, that the diffusion matrix $D=C_S$ is positive definite, and hence coercive. In general, $m_{HC}$ is finite when we are assuming Condition $A$ (see Lemma $2.3$, \cite{AE}).
\end{remark}
For completeness, we \bea{include} the definition of hypocoercivity index also for the non-normalized case. For simplicity we will denote it as well with $m_{HC}$. This is actually allowed since the next proposition will prove that these two definitions are unchanged under normalization.
\begin{definition}
We define $m_{HC}$ the \textit{hypocoercivity index} for the FP-equation \eqref{FP} as the minimum $m  \in \mathbb{N}_0$ such that 
\begin{equation}
\widetilde{T}_{m}:=\sum_{j=0}^m \widetilde{C}^j \widetilde{D} (\widetilde{C}^T)^j >0,
\end{equation}
\bea{and $m_{HC}=\infty$ if this minimum does not exist.}
\end{definition}
\begin{prop}
Let us consider the FP-equation \eqref{FP} and its normalized version \eqref{NormalizedFPwithC}. Let Condition $\widetilde{A}$ (or, equivalently, Condition $A$) be satisfied. Then, the hypocoercivity indices of the two equations coincide, i.e., for any $m \in \mathbb{N}_0$
\begin{equation}
\label{Tmetilde}
T_m>0 \text{ \quad if and only if \quad } \widetilde{T}_m >0.
\end{equation}
\end{prop}
\begin{proof}
First we \bea{recall from Lemma 2.3, \cite{AAC} that
\begin{equation}
\label{step1MHC}
\sum_{j=0}^m C_{AS}^j C_S (C_{AS}^T)^j >0 \text{\quad if and only if \quad} \sum_{j=0}^m C^j C_S (C^T)^j >0.
\end{equation}}

The second step consists in proving that $\widetilde{T}_m >0$ \  iff 
$$
  \bea{\hat{T}}_m:=\sum_{j=0}^m C^j D (C^T)^j >0,
$$
where $C=K^{-1/2}\widetilde{C}K^{1/2}$ and $D=K^{-1/2}\widetilde{D} K^{-1/2}=C_S$ are the matrices appearing in the normalized equation and $K$ from \eqref{Lyapeq}. By substituting we get
\begin{align*}
\bea{\hat{T}}_m=&\sum_{j=0}^m (K^{-1/2}\widetilde{C}K^{1/2})^j K^{-1/2}\widetilde{D} K^{-1/2} (K^{1/2}\widetilde{C}^TK^{-1/2})^j
\\
=&K^{-1/2} \sum_{j=0}^m \widetilde{C}^j \widetilde{D} (\widetilde{C}^T)^j K^{-1/2}
\\
=&K^{-1/2}\widetilde{T}_m K^{-1/2}.
\end{align*}
Then, it is immediate to conclude that the positivity of the two matrices is equivalent since $K>0$.

Combining this last equivalence with \eqref{step1MHC} yields \eqref{Tmetilde}.
\end{proof}

\begin{remark}
We shall now compare the hypocoercivity index $m_{HC}$ of the normalized FP-equation \eqref{NormalizedFPwithC} to the commutator condition $(3.5)$ in \cite{Vi09}. To this end we rewrite \eqref{NormalizedFPwithC} for $h(x,t):=f(x,t)/f_{\infty}(x)$. In H\"ormander form it reads 
\bea{
\begin{equation}\label{hoermander-form}
\partial_t h=\mathrm{div}(C\nabla h)-x^TC\nabla h= -(A^*A+B)h, 
\end{equation}
}
where the adjoint $A^*$ is taken w.r.t. $L^2(\R^d, f_{\infty})$. Here, the vector valued operator $A$ and the scalar operator $B$ are given by 
$$
A:=\sqrt{D} \cdot \nabla, \qquad B:=x^T \cdot C_{AS} \cdot \nabla.
$$
Following \S 3.3 in \cite{Vi09} we define the iterated commutators 
$$
C_0:=A, \qquad C_k:=[C_{k-1},B].
$$
They are vector valued operators mapping from \bea{$L^2(\R^d,f_{\infty})$ to $(L^2(\R^d,f_{\infty}))^d$}. Hence, the nabla operator in $B$ can be either the gradient or the Jacobian, depending on the dimensionality of the argument of $B$.
\bea{By induction one} easily verifies that $C_k=\sqrt{D}\cdot C_{AS}^k \cdot \nabla$, $k\in \N_0$.

We recall condition $(3.5)$ from \cite{Vi09}: ``There exists $N_c \in \N_0$ such that
\begin{equation}
\label{CondVillani}
\sum_{k=0}^{N_c} C_k^*C_k \text{\quad is coercive on } \mathrm{ker}(A^*A+B)^{\perp}.\: \mbox{''}
\end{equation}
Note that  $\mathrm{ker}(A^*A+B)$ consists of the constant functions, and its orthogonal \bea{complement} is $\{ h \in L^2(\R^d,f_{\infty}) \ : \ \int_{\R^d} h f_{\infty} dx=0 \}.$ The coercivity in \eqref{CondVillani} reads 
\begin{equation}
\label{WPIneq}
\int_{\R^d} \nabla^T h \cdot T_{N_c} \cdot \nabla h  f_{\infty} dx \geq \kappa \int_{\R^d} h^2 f_{\infty} dx
\end{equation}
for some $\kappa>0$ and all $h \in  \mathrm{ker}(A^*A+B)^{\perp}$, where $T_{N_c}:= \sum_{k=0}^{N_c} (C^T_{AS})^k D C_{AS}^k$. Clearly, the weighted Poincar\'e inequality \eqref{WPIneq} holds iff $T_{N_c}>0$, see \S 3.2 in \cite{AMTU}, e.g. Hence, the minimum $N_c$ for condition \eqref{CondVillani} to hold equals the hypocoercivity index $m_{HC}$ from Definition \ref{defHCI1} above.
\end{remark}

Next we shall link the hypocoercivity index of the FP-equation with the hypocoercivity index $m_{HC}$ of its associated ODE $\dot{x}(t)=-Cx(t)$, which is defined in the same way.
At the ODE level, this index describes the short time decay of the propagator norm $\left\|e^{-Ct}\right\|_{\mathcal{B}(\R^d)}$ as it is shown in the following Theorem \ref{shortTimeODE} (see Theorem 2.6, \cite{AAC}).

\bea{
\begin{remark}
We note that our hypocoercivity index $m_{HC}$ also coincides with the index appearing in the characterization of the \emph{singular space} $S$ of the FP-operator, i.e.\ the smallest integer $k_0$ such that
$$
  \bigcap_{j=0}^{k_0} \mathrm{ker}[ C_S(C_{AS})^j] = S = \{0\}
$$ 
(see (2.9) in \cite{AB}, (3.22) in \cite{OPP}). The equivalence of these two indices follows since they are both equivalent to the smallest integer $\tau$ in the \emph{Kalman rank condition}, i.e.\ 
$$
  \mathrm{rank} \big\{\sqrt{C_S},\, C_{AS}\sqrt{C_S},\, ...,C_{AS}^\tau\sqrt{C_S}\big\}=d\ .
$$
This was established in Proposition 1 of \cite{AAC18} and, respectively, on pages 705/706 of \cite{OPP}. The latter proof uses the version \eqref{hoermander-form} of the FP-equation.
\end{remark}
}

\begin{theorem}
\label{shortTimeODE}
Let $C$ satisfy Condition $A$. Then its  hypocoercivity index is $m_{HC} \in \N_0$ \bea{(and hence finite)} if and only if 
\begin{equation}\label{EQshortTimeODE}
   \left\|e^{-Ct}\right\|_{\mathcal{B}(\R^d)} = 1 - ct^{\alpha}+\mathcal{O}(t^{\alpha+1}), \qquad\mbox{as } t \rightarrow 0+ \,,
\end{equation}
for some $c>0$, where $\alpha:=2 m_{HC}+1$.
\end{theorem}

\begin{remark}
We observe that, in the coercive case (i.e., $m_{HC}=0$), the propagator norm satisfies an estimate of the form
\begin{equation} 
\label{COERC}
\left\|e^{-Ct}\right\|_{\mathcal{B}(\R^d)}\leq e^{-\lambda t }, \quad t \geq 0, \text{ for some } \lambda>0.
\end{equation}
In that case ($\alpha=1$) Theorem \ref{shortTimeODE} states that the propagator norm \linebreak 
$\left\|e^{-Ct}\right\|_{\mathcal{B}(\R^d)}$ behaves as $g(t):=1-ct$ for short times. With $c=\lambda$, this is the (initial part of the) Taylor expansion of the exponential function in \eqref{COERC}.
\end{remark}
Next we shall use this result to derive information about the short time behavior of the Fokker-Planck propagator norm $\| e^{-Lt}\|_{\mathcal{B}(V_0^{\perp})}$. By Theorem \ref{maintheo} the propagator norms of the FP-equation and the corresponding ODE coincide. 

\begin{theorem}
\label{TheoShortFP}
Let $L$ be the Fokker-Planck operator defined in \eqref{NormalizedFPwithC}. Let $C$ satisfy Condition $A$. Then \bea{the hypocoercivity index} of \eqref{NormalizedFPwithC} is $m_{HC}\in \mathbb{N}_0$ \bea{(and hence finite)} if and only if 
\begin{equation}
\label{eqshortFP}
\left\|e^{-Lt}\right\|_{\mathcal{B}(V_0^{\perp})} =1-ct^{\alpha}+\mathcal{O}(t^{\alpha+1}), \quad t \rightarrow 0+,
\end{equation}
where $\alpha=2 m_{HC}+1$, for some $c >0.$
\end{theorem}

\begin{proof}
This result is an immediate corollary of Theorem \ref{maintheo} and Theorem \ref{shortTimeODE}, by recalling that the FP-equation and its associated ODE have the same hypocoercivity index.
\end{proof}
\begin{remark}
As for the ODE case, the equality \eqref{eqshortFP} shows that the index $m_{HC}$ describes how fast the propagator norm decays for short times. This is consistent with the fact that the coercive case ($m_{HC}=0$) corresponds to the fastest behavior, i.e., with an exponential decay ($\alpha=1$). In general, the bigger the index, the slower is the decay of the norm for short times.
\end{remark}

\begin{example}
In Theorem $1.2$ of \cite{GM} the authors derive the \bea{explicit expression} for the propagator norm of the FP-equation associated to the matrix \eqref{Ca}, see Theorem \ref{theoGM}. \bea{With it they also estimate} the short time behavior of this norm, depending on the parameter $a$. In the case $a>0$, equality $(2)$ in \cite{GM} implies 
$$
\left\|e^{-\widetilde L_a t}\right\|_{\mathcal{B}(V_0^{\perp})}=1-\frac{a}{6}t^3+o(t^3).
$$
We note that this result is consistent with the equality \eqref{eqshortFP}. Indeed, it is easy to verify that for $a>0$ the matrix $C_a$ has hypocoercivity index $m_{HC}=1$. Hence the exponent in the polynomial short time behavior turns out to be $\alpha=3$, as above. $ \hfill \qed$
\end{example}

\bea{It is known that} the hypocoercivity index \bea{also has} a second implication on the qualitative behavior of FP-equations, namely the rate of regularization from some weighted $L^2$-space into a weighted $H^1$-space (like in non-degenerate parabolic equations). The following proposition was proven in \cite{Vi09} (see \S 7.3, \S A.21 for the kinetic FP-equation with $m_{HC}=1$. The extension from Theorem A.12 is given without proof and includes a small typo.) and in \cite[Theorem 4.8]{AE}. \bea{The following result can also be seen as a special case of (2.21) as well as of Theorem 2.6 in \cite{AB}.}
\begin{prop}
\label{RegWitha}
Let $f(t)$ be the solution of \eqref{NormalizedFPwithC}. Let $C$ satisfy Condition $A$ and $m_{HC}$ be its associated hypocoercivity index. Then, there exist $\tilde{c}$, $\delta >0$, such that
\begin{equation}
\label{RegH1}
   \left\| f_\infty\nabla\left(\frac{f(t)}{f_{\infty}}\right )\right\|_{\mathcal{H}} \leq \tilde{c} t^{-\alpha/2}
   \left\|f_0\right\|_{\mathcal{H}}, \qquad 0<t\leq \delta,
\end{equation}
with $\alpha:=2 m_{HC}+1$ for all $f_0 \in \mathcal{H}$.
\end{prop}
So far we have seen that the hypocoercivity index of a FP-equation determines both the short time decay and its regularization rate. An obvious question is now to understand the relation of these two qualitative properties. The following proposition shows that they are essentially equivalent for the family \eqref{NormalizedFPwithC} of FP-equations:
\begin{prop}
\label{Prop319}Let \bea{the matrix} $C$ satisfy Condition $A$ \bea{(see Definition \ref{conA})},  
and let $f(t)$ be the solution of \eqref{NormalizedFPwithC}. We denote its propagator norm by $\left\|e^{-Lt}\right\|_{\mathcal{B}(V_0^{\perp})}=:\tilde{h}(t)$, $t \geq 0$. 
 \begin{enumerate}[(a)]
 \item Assume that $\tilde{h}(t)=1-ct^{\alpha}+o(t^{\alpha})$ as $t \rightarrow 0^+$ for some $c>0$ and $\alpha >0$. Then the regularization estimate \eqref{RegH1} follows with the same $\alpha$, and for all $f_0\in \mathcal{H}$.
 Moreover, this $\alpha$ in \eqref{RegH1} is optimal (i.e.\ minimal).
 \item Let there exist some $\tilde{c}, \delta >0$ and $\alpha > 0$ (not necessarily integer) such that 
 \eqref{RegH1} holds \bea{for all} $f_0 \in \mathcal{H}$. Then, \bea{there are $\delta_2>0$ and $c_2>0$, such that} $\tilde{h}(t) \leq 1- c_2 t^{\alpha}$ on $0 \leq t \leq \delta_2$.
 Moreover, if $\alpha$ is minimal in the assumed regularization estimate \eqref{RegH1}, then it is also minimal in the concluded decay estimate $\tilde{h}(t) \leq 1- c_2 t^{\alpha}$.
 \end{enumerate}
\end{prop}
The proof of Proposition \ref{Prop319} can be found in the Appendix, since it requires results that will be presented in the next sections. \\

\bea{\begin{remark}
We note that the statements \eqref{eqshortFP} and \eqref{RegH1} are different in nature: While the equality \eqref{eqshortFP} \emph{characterizes} the short-time decay of $e^{-Lt}$, the inequality \eqref{RegH1} only provides an \emph{upper bound} for the short time regularization of $e^{-Lt}$. Hence, since Proposition \ref{RegWitha} is based on \eqref{RegH1}, it can only yield the conclusion $\tilde{h}(t) \leq 1-c_2 t^{\alpha}$, which is also just an upper bound for the short time behavior, rather than the dominant part of the Taylor expansion of $\tilde{h}(t)$. But if $\alpha$ is known to be minimal in \eqref{RegH1}, then it is also minimal for \eqref{eqshortFP}.
\end{remark}}
\begin{remark}
Proposition \ref{RegWitha} provides an \emph{isotropic} regularization rate. We note that this result can be improved for degenerate, hypocoercive FP-equations, \bea{and it gives rise} to anisotropic smoothing: There the regularization is faster in the diffusive directions of $(\ker C_S)^\perp$ than in the non-diffusive directions of $\ker C_S$. ``Faster'' corresponds here to a smaller exponent in \eqref{RegH1}.

An example of different speeds of regularization is given in \cite[Section 11]{Sc} for the solution $f(t, x,v)$ of a kinetic FP-equation in $\mathbb{T}^d \times \R^d$ without confinement potential. In that case the short-time regularization estimate for the $v$-derivatives is the same as for the heat equation, since the operator is elliptic in $v$. But the regularization in $x$ has an exponent 3 times as large; this corresponds, respectively, to the two cases $m_{HC}=0,\,1$ in \eqref{RegH1}. 
A more general result about anisotropic regularity estimates can be found in \cite[Section A.21.2]{Vi09}. 
In an alternative description one can fix a uniform regularization rate in time, by considering different regularization orders (i.e.\ higher order derivatives) in different spatial directions in the setting of \emph{anisotropic} Sobolev spaces. A definition of these functional spaces and an example of this behaviour is provided in \cite{OPP}, regarding the solution of a degenerate Ornstein-Uhlenbeck equation.
\end{remark}
\end{subsection}
\end{section}
\section{Solution of the FP-equation by spectral decomposition}
In order to link the evolution in (\ref{NormalizedFPwithC}) to the corresponding drift ODE $\dot x=-Cx$ we shall project the solution $f(t)\in \mathcal{H}$ of \bea{\eqref{NormalizedFPwithC}} to finite dimensional subspaces  $\{V^{(m)} \}_{m \in \N_0} \subset \mathcal{H}$ with $L V^{(m)} \subseteq V^{(m)}$. Then we shall show that, surprisingly, the evolution in each subspace can be based on the single ODE $\dot x=-C x$. 
\bea{Moreover, the solution component in the subspace $V^{(1)}$ will turn out to decay the slowest, and it is hence the dominant part.}
\subsection{Spectral decomposition of the Fokker Planck operator}
\label{sec:4.1}
First we define the finite dimensional, $L$-invariant subspaces $V^{(m)} \subset \mathcal{H}$.
Let the dimension $d \geq 1$ be fixed. From \cref{intro} we recall that the (normalized) steady state of \eqref{NormalizedFPwithC} is given by $g_0(x):=f_{\infty}\bea{(x)}= \prod_{i=1}^d g(x_i)$, $x=(x_1,\ldots,x_d) \in \R^d$, where $g(y)=\frac{1}{\sqrt{ 2 \pi}}e^{-y^2/2}$ is the one-dimensional (normalized) Gaussian.
The construction and results about the spectral decomposition of $L$ that we are going to summarize can be found in \cite[Section 5]{AE}.

\begin{definition}
Let $\alpha=(\alpha_i) \in \N_0^d$ be a multi-index. Its order is denoted by $|\alpha|=\sum_{i=1}^d \alpha_i$.
For a fixed $\alpha \in \mathbb{N}_0^d$ we define
\begin{equation}
\label{galpha}
g_{\alpha}(x):= (-1)^{|\alpha|}\nabla ^{\alpha}_x g_0(x),
\end{equation}
or, equivalently, 
\begin{equation}
\label{hermiteingalpha}
g_{\alpha}(x):=\prod_{i=1}^d  H_{\alpha_i}(x_i) g(x_i), \quad  \forall x=(x_i) \in \R^d,
\end{equation}
where, for any $n \in \mathbb N_0$, $H_n$ is the \textit{probabilists' Hermite polynomial} of order $n$ defined as 
$$
H_n(y):=(-1)^n e^{\frac{y^2}{2}} \frac{d^n}{dy^n} e^{-\frac{y^2}{2}}, \quad \forall y \in \R.
$$
\end{definition}

\begin{lemma}
\label{normofgalpha}
Let $\alpha=(\alpha_i) \in \N_0^d $.  Then,
\begin{equation}
\label{normofgalphaEQ}
\|g_{\alpha}\|_{\mathcal{H}}=\sqrt{ \alpha !} = \sqrt{\alpha_1 ! \cdots \alpha_d !}\ .
\end{equation}
\end{lemma}
\begin{proof}
We compute
\begin{align*}
\|g_{\alpha}\|^2_{\mathcal{H}}\bea{=} 
\int_{\R^d} \prod_{i=1}^d H_{\alpha_i}(x_i)^2 g(x_i)^2 g(x_i)^{-1} dx
 = \prod_{i=1}^d \int_{\R} H_{\alpha_i}(x_i)^2 g(x_i) dx_i=\prod_{i=1}^d \alpha_i!\ ,
\end{align*}
where we have used the following weighted $L^2$-norm of $H_n$:
\begin{equation}
\label{wnormofhermite}
\int_{\R} H_n(y)^2 g(y)\,dy=n! \,.
\end{equation}
\end{proof}

\begin{definition}
We define the index sets $S^{(m)}:=\{\alpha \in \N_0^d:|\alpha|=m \}$, $m \in \N_0$. For any $m \in \mathbb{N}_0$, the subspace $V^{(m)}$ of $\mathcal{H}$ is defined as
\begin{equation}
\label{Vmdef}
V^{(m)}:=\operatorname{span}_{\R}\left\{g_{\alpha}: \ \alpha \in S^{(m)} \right\} \,.
\end{equation}
\end{definition}
\begin{remark}
$V^{(m)}$ has dimension
\begin{equation}
\label{eq:Gamma_m}
 \Gamma_m:=|S^{(m)}|=\binom{d+m-1}{m}< \infty .
 \end{equation}
 \end{remark}
 Let us consider some examples. If $d=2$ we have
\begin{enumerate}
\item $V^{(0)}=\{ \beta_1 g_0(x) , \beta_1 \in \R \}$;
\item $V^{(1)}=\operatorname{span}{ \{ g_{(1,0)}, g_{(0,1)}    \}}= \operatorname{span}\left\{  x_1e^{-|x|^2/2}, \ x_2e^{-|x|^2/2} \right\}
\\
=\{(\beta_1x_1+\beta_2x_2)g_0(x), \ \beta_1,\beta_2 \in \R \}$;
\item $V^{(2)}= \operatorname{span}{ \{g_{(2,0)}, g_{(1,1)}, g_{(0,2)} \}}
\\
=\left \{ \left[\beta_1(x_1^2-1)+\beta_2 x_1x_2 +\beta_3(x_2^2-1)\right]g_0(x), \ \beta_i \in \R, \ i=1,2,3 \right \};
$
\item $V^{(3)}=\operatorname{span}{ \{ 
g_{(3,0)}, g_{(2,1)}, g_{(1,2)}, g_{(0,3)}
\}}
\\
=\big \{
\left[\beta_1\bea{(x_1^3-3x_1)+\beta_2(x_1^2x_2-x_2)+\beta_3(x_2^2x_1-x_1)+\beta_4(x_2^3-3x_2)}\right]g_0(x), 
\\
\beta_1,...,\beta_4 \in \R
\big \}.
$
\end{enumerate}
It is well known that $\{ g_{\alpha} \}_{\alpha \in \N_0^d}$ forms an orthogonal basis of $\mathcal{H} = L^2(\R^d,g_0^{-1})$. Hence, also the subspaces $V^{(m)}$ are mutually orthogonal.
This yields an orthogonal decomposition of the Hilbert space 
\begin{equation}
\label{Hdecompo}
\mathcal{H}= \bigoplus_{m \in \N_0} {}^{\bea{\!\!\!\!\!\perp}}\:\, V^{(m)}.
\end{equation}
\begin{remark}
In \cite[\S 5]{LNP} an alternative block diagonal decomposition of the FP-\bea{propagator} (when considered in the flat $L^2(\R^d)$) into finite-dimen\-sional subspaces is derived by using \textit{Wick quantization}.
\end{remark}
We also \bea{consider} the normalized version of the basis elements of the subspaces $V^{(m)}$:
\begin{definition}[Normalized basis]
For each fixed $\alpha \in \mathbb{N}_0^d$, we denote with $\tilde{g}_{\alpha}$ the normalized function 
$$\tilde{g}_{\alpha}:=\frac{g_{\alpha}}{\|g_{\alpha}\|_{\mathcal{H}}}.$$
\end{definition}

The reason why we need both $g_{\alpha}$ and $\tilde{g}_{\alpha}$ is that we can obtain a "nicer" evolution of $f(t)$ projected into $V^{(m)}$ in terms of the matrix $C$ with the first ones. Instead, the functions $\tilde{g}_{\alpha}$ can be used \bea{to express} the equivalence of norms by Plancherel's equality in the Hilbert space $\mathcal{H}$.

The orthogonal decomposition \eqref{Hdecompo} \bea{allows to express $f(t) \in L^2(\R^2, f_{\infty}^{-1})$, for a fixed $t \geq 0$, in the form}
\begin{equation}
\label{fdecomposed}
f(t,x)=\sum_{\alpha \in \N_0^d} \frac{\langle f(t),g_{\alpha} \rangle _{\mathcal{H}}}{\|g_{\alpha}\|^2_{\mathcal{H}}}g_{\alpha}(x)=:\sum_{\alpha \in \N_0^d} d_{\alpha}(t)g_{\alpha}(x),
\end{equation}
or in terms of the normalized basis,
\begin{equation}
\label{finVtildem}
f(t,x)=\sum_{\alpha \in \N_0^d} \langle f(t),\tilde{g}_{\alpha} \rangle _{\mathcal{H}}\tilde{g}_{\alpha}(x)=:\sum_{\alpha \in \N_0^d} \tilde{d}_{\alpha}(t)\tilde{g}_{\alpha}(x).
\end{equation}
The Fourier coefficients corresponding to a subspace $V^{(m)}$ \bea{can be} grouped into vectors \bea{in $\R^{\Gamma_m}$}:
$$
   d^{(m)} := \left(d_{\alpha}\right)_{\alpha \in S^{(m)}},  \bea{\text{ and \ }}  \tilde{d}^{(m)} := \left(\tilde{d}_{\alpha}\right)_{\alpha \in S^{(m)}}.
$$
\bea{By the completeness of the Hilbert orthonormal basis $\{\tilde{g}_{\alpha}\}_{\alpha \in \N_0^d}$ in $\mathcal{H}$,} Plancherel's Theorem then yields
\begin{equation}
\label{decompoff}
\|f\|_{\mathcal{H}}^{2}=\sum_{m\geq 0}\left\|\tilde{d}^{(m)}\right\|_2^2=\sum_{m \geq 0} \sum_{\alpha \in S^{(m)}} |\tilde{d}_{\alpha}|^2=\sum_{m \geq 0} \sum_{\alpha \in S^{(m)}} |d_{\alpha}|^2\|g_{\alpha}\|_{\mathcal{H}}^2,
\end{equation}
where we have used the relation $\tilde{d}_{\alpha}=\|g_{\alpha}\|_{\mathcal{H}}d_{\alpha}$.

Moreover, we denote by $(\Pi_mf)\in V^{(m)}$ the orthogonal \textit{projection of $f$ into $V^{(m)}$}. It is given by
$$
(\Pi_m f)= \sum_{\alpha \in S^{(m)}} d_{\alpha}g_{\alpha} =\sum_{\alpha \in S^{(m)} } \widetilde{d}_{\alpha} \widetilde g_{\alpha} \,.
$$
It follows that 
\begin{equation}
\label{linkProjVect}
\left\| \Pi_m f\right\|_{\mathcal{H}}=\left\|\tilde{d}^{(m)}\right\|_2 \,.
\end{equation}
In the next proposition we shall see that the subspaces $V^{(m)}$ are invariant under the action of the operator $L$, by giving the explicit action of $L$ on each basis element $g_{\alpha}$. For this purpose we introduce a notation for shifted multi-indices. 
\begin{definition}
Given $\alpha=(\alpha_i) \in \N_0^d$ and $l\in \langle d\rangle:=\{1,...,d \}$, we define the components of the multi-indices $\alpha^{(l-)},\,\alpha^{(l+)} \in \N_0^d$ as
$$
   \alpha^{(l\pm)}_j :=\alpha_j \quad \mbox{for } j \neq l \,, \qquad \alpha^{(l\pm)}_l :=(\alpha_l \pm1)_+ \,.
$$
\end{definition}
So, for instance, if $g_{\alpha}\in V^{(m)}$ and $\alpha_l> 0$, then $g_{\alpha^{(l-)}}\in V^{(m-1)}$ and
\linebreak
 $g_{(\alpha^{(l-)})^{(j+)}} \in V^{(m)}$. Note that cutting off negative values guarantees that $\alpha^{(l-)}$ is always an admissible multi-index.
This part of the definition will, however, not influence the following.

The action of the operator $L$ on $V^{(m)}$ \bea{can be} taken from \cite[Proposition 5.1 and its proof]{AE}:
\begin{prop}
\label{propVminv}
For every $m \in \N_0$, the subspace $V^{(m)}$ is invariant under $L$, its adjoint $L^*$ and, hence, the solution operator \bea{$e^{-Lt}$}, $t\geq 0$.
Moreover, for each $g_{\alpha}$, 
\begin{equation}
\label{LVm}
Lg_{\alpha}=-\sum_{j,l =1}^d \alpha_l C_{jl} g_{(\alpha^{(l-)})^{(j+)}}\ ,
\end{equation}
where $C_{jl}$ are the matrix elements of  \,$C$.
\end{prop}


\subsection{Evolution of the Fourier coefficients}
In this section we shall derive the evolution of $\Pi_m f$ in terms of the Fourier coefficients $d^{(m)}$:
\begin{prop}
\label{formulaevolVm}
Let $f$ satisfy the FP-equation \eqref{NormalizedFPwithC}. Then the coefficients in the expansion \eqref{fdecomposed}
satisfy
\begin{equation}
\label{evolutiondm}
\bea{\frac{d}{dt}}d_{\alpha}= -\sum_{j,l=1}^{d} \mathbb{1}_{\alpha_j\ge 1}(\alpha^{{(j-)}})^{(l+)}_l C_{jl} d_{(\alpha^{(j-)})^{(l+)}} \,,\qquad 
\alpha \in \N_0^d\,.
\end{equation}
\end{prop}
\begin{proof}
We substitute \eqref{fdecomposed} into \eqref{NormalizedFPwithC} and use \eqref{LVm}:
$$
   \sum_{\alpha \in \N_0^d} \bea{\frac{d}{dt}}d_{\alpha} g_\alpha = - \sum_{j,l =1}^d \sum_{\alpha:\, \alpha_l\ge 1} d_\alpha 
    \alpha_l C_{jl} g_{(\alpha^{(l-)})^{(j+)}}  \,.
 $$
 In the sum over $\alpha$ on the right hand side we substitute 
 $$
      (\alpha^{(l-)})^{(j+)}=\beta \quad\Longleftrightarrow\quad \alpha = (\beta^{(j-)})^{(l+)} \,,
 $$
 leading to
 \begin{eqnarray*}
   \sum_{\alpha \in \N_0^d} \bea{\frac{d}{dt}} d_{\alpha} g_\alpha &=& - \sum_{j,l =1}^d \sum_{\beta:\, \beta_j\ge 1} d_{(\beta^{(j-)})^{(l+)}} 
    (\beta^{(j-)})^{(l+)}_l C_{jl} g_\beta \\
    &=&  \sum_{\beta \in \N_0^d} \left( -\sum_{j,l=1}^{d} \mathbb{1}_{\beta_j\ge 1}(\beta^{{(j-)}})^{(l+)}_l C_{jl} 
     d_{(\beta^{(j-)})^{(l+)}} \right) g_\beta\,,
\end{eqnarray*}
completing the proof.
\end{proof}
\bea{
\begin{remark}\label{matrix-Cm}
{}From the family of equations \eqref{evolutiondm} we can deduce: The vector $d^{(m)}=(d_{\alpha})_{\alpha \in S^{(m)}} \in \R^{\Gamma_m}$ satisfies the ODE  $\frac{d}{dt}d^{(m)}=-C^{(m)} d^{(m)}$ for some matrix $C^{(m)} \in \R^{\Gamma_m \times \Gamma_m}$. Actually, we shall not write down the matrix $C^{(m)}$ explicitly, as we shall not need it. 
\end{remark}}
As the simplest example we shall first consider the evolution in $V^{(1)}$. We use the notation 
$S^{(1)} = \{\alpha(1),\ldots, \alpha(d)\}$ with $\alpha(k)_j = \delta_{jk}$, $j,k=1,\ldots,d$. In the right hand side of \eqref{evolutiondm} with $\alpha=\alpha(k)$ obviously only the terms with $j=k$ are nonzero, 
$(\alpha(k)^{{(k-)}})^{(l+)}=\alpha(l)$ and, thus, $(\alpha(k)^{{(k-)}})^{(l+)}_l=1$. This implies
$$
 \bea{ \frac{d}{dt}} d_{\alpha}= - \sum_{l=1}^d C_{kl} d_{\alpha(l)}
$$
and therefore 
\begin{equation}\label{d1ev}
 \bea{\frac{d}{dt}}d^{(1)}= -C d^{(1)} \qquad\mbox{for } d^{(1)} = \left(d_{\alpha(1)},\ldots,d_{\alpha(d)}\right) \,.
\end{equation}
We define $h(t):=\left\|e^{-Ct}\right\|_{\mathcal{B}(\R^d)}$. Then \eqref{d1ev} implies 
\begin{equation}
\label{d1}
h(t) = \sup_{ 0 \neq \tilde{d}^{(1)}(0) \in \R^{\Gamma_1}}   \frac{\|\tilde{d}^{(1)}(t)\|_2 }{ \|\tilde{d}^{(1)}(0)\|_2} \,, \qquad t \geq 0 \,.
\end{equation}
To analyze the evolution in $V^{(m)}$, $m \geq2$, it turns out that the representation of $d^{(m)}$ as a vector is not convenient. In the next section we shall rather represent it as a tensor. Not as a tensor of order $d$,
as the number of components of $\alpha$ would indicate, but as a symmetric tensor of order $m$ over $\R^d$.
This way it will be easier to characterize its evolution -- in fact as a tensored version of \eqref{d1ev}.


\section{Subspace evolution in terms of tensors}

\subsection{Order-$m$ tensors}
In this subsection we briefly review some notations and basic results on tensors that will be needed. Most of their elementary proofs are deferred to the appendix. For more details we refer the reader to \cite{CGLM} and \cite{L}.

Let $m\in\N$ be fixed. \bea{We note that along the paper the convention $\N=\{1,2,...\}$, excluding zero, is used}.
\begin{definition}
For $n_1,...,n_m \in \mathbb{N}$, a function $h: \langle n_1 \rangle \times \cdots \times \langle n_m \rangle  \rightarrow \R$ is a (real valued) \textit{hypermatrix}, also called \textit{order-$m$ tensor } or \textit{$m$-tensor}, where $\langle n_k \rangle:=\{1,...,n_{k} \}$, $\forall 1 \leq k \leq m$. 
We denote the set of values of $h$ by an $m$-dimensional table of values, calling it $A=(A_{{i_1}...{i_m}})_{i_1,...,i_m=1}^{n_1,...,n_m},$ or just $A=(A_{{i_1}...{i_m}})$. 
The set of order-$m$ hypermatrices (with domain $\langle n_1 \rangle \times \cdot \cdot \cdot \times  \langle n_m \rangle $) is denoted by $T^{n_1 \times \cdots \times n_m}$.

We will consider only the case in which $n_1=\cdots=n_m=d$, i.e., $A=(A_{i_1...i_m})_{i_1,...,i_m=1}^d$.
In this case, we will denote $T^{(m)}_d:=T^{d \times \cdot \cdot \cdot \times d}$ for simplicity. Also, since in our case the dimension $d$ is fixed, we will denote it by $T^{(m)}$.
Then $A \in T^{(m)}$ is a function from $\langle d\rangle^m$ to $\R$, denoted by $A=(A_{I})_{I \in \langle d\rangle^m}$.

\end{definition}
It will be useful to define some operations on $T^{(m)}_d$:
\begin{definition}
It is natural to define the operations of entrywise addition and scalar multiplication that make $T^{(m)}$ a vector space in the following way: for any $A,B \in T^{(m)}$ and $\gamma \in \R$
$$(A+B)_{i_1...i_m}:=A_{i_1...i_m}+B_{i_1...i_m}, \quad  (\gamma A)_{i_1...i_m}:= \gamma A_{i_1...i_m}.$$
Moreover, given $m$ matrices $B_1=(b^{(1)}_{ij}), ... , B_m=(b^{(m)}_{ij}) \in \R^{d \times d}=T^{(2)}$ and $A \in T^{(m)}$, we define the \textit{multilinear matrix multiplication} by 
\\
$A':=(B_1,...,B_m)\odot A \in T^{(m)}$ where
\begin{equation}
A'_{i_1...i_m}:= \sum_{j_1,...,j_m=1}^d b^{(1)}_{i_1j_1} \cdot \cdot \cdot b^{(m)}_{i_mj_m} A_{j_1...j_m}.
\end{equation}
For $A \in T^{(m)}$ and $k\leq m $ matrices $B_1,...,B_k \in T^{(2)}$, we also define the product    $A':=(B_1,...,B_k)\odot A \in T^{(m)}_d$ in the following way:
$$
A'_{i_1...i_m}:= \sum_{j_1,...,j_k=1}^d b^{(1)}_{i_1j_1} \cdot \cdot \cdot b^{(k)}_{i_kj_k} A_{j_1...j_ki_{k+1}...i_{m}},
$$
i.e., the multiplication acts on the first $k$-indices of $A$.
For simplicity, when $B_1=...=B_k:=B$, we will denote $(B_1,...,B_k) \odot A$ by $B \odot^{k} A$.
For example, if $d=4$ and given $B=(b_{ij}) \in \R^{4\times 4}, A\in T^{(3)}$,
$$
(B \odot A)_{i_1i_2i_3}= \sum_{j=1}^{4} b_{i_1j} A_{ji_2i_3},
$$
and 
$$
B \odot^{3} A=(B,B,B) \odot A.
$$
\end{definition}
Finally, we equip $T^{(m)}$ with an inner product:
\begin{definition}
Let $A=(A_{i_1...i_m}), B=(B_{i_1...i_m}) \in T^{(m)}$, we call $\langle A,B \rangle_{\mathcal{F}} \in \R$ the \textit{Frobenius inner product} between the $m$-tensors $A$ and $B$, defined by
$$
 \langle A,B \rangle _{\mathcal{F}}:= \sum_{i_1,...,i_m=1}^d  A_{i_1...i_m}B_{i_1...i_m}.
$$
This induces a norm in $T^{(m)}$, called \textit{Frobenius norm} in the natural way: 
$$\|A\|_{\mathcal{F}}:=\sqrt{ \langle A,A \rangle _{\mathcal{F}}}=\left(\sum_{i_1,...,i_m=1}^d(A_{i_1...i_m})^2\right)^{1/2} \geq 0 .
$$
\end{definition}

\begin{definition}
The tensor $D=(D_{I})_{I \in \langle d\rangle^m} \in T^{(m)}$ is called \textit{symmetric}, if $\forall I \in$ $\langle d\rangle^m$ it is true that $D_{I}=D_{\sigma(I)}$ for every permutation $\sigma$ \bea{of $m$ elements}.
\bea{Then} $F^{(m)} \subset T^{(m)}$ (and occasionally $F^{(m)}_d$) denotes the set of symmetric $m$-tensors.
Given $A\in T^{(m)}$, we define the \textit{symmetric part of $A$} as the symmetric tensor defined by
$$
\mathrm{Sym}{A}:= \frac{1}{m!} \sum_{ \sigma \in \mathcal{P}} \sigma(A) \in F^{(m)}, 
$$
where $\mathcal{P}$ is the \bea{group} of permutations \bea{of $m$ elements} and $\sigma(A)$ is the tensor with components $\sigma(A)_I:=A_{\sigma(I)}$, $\forall I \in \langle d\rangle^m$.
\end{definition}

\begin{remark}
\label{identsymm}
For a symmetric tensor $D \in F^{(m)}$, clearly we do not need to define $D_{I}$ for each $I=(i_1,...,\bea{i_m})\in \langle d\rangle^m$ since the value of $D_{I}$ depends only on the number of occurrences of each value in the index $I$. Therefore,  we define the function $\vphi:\,\langle d\rangle^m\to S^{(m)}$ with
$$
  \vphi_k(I):= \sum_{j=1}^{m} \bea{\delta_{k,i_j}},\quad 
  \forall k=1,...,d \quad \mbox{and for each}\quad I=(i_1,...,i_m) \in \langle d\rangle^m\ ,
$$
\bea{where $\delta_{k,i}$ denotes the Kronecker symbol.}
Hence, the component $\vphi_k$ counts the occurrences of $k$ in the multi-index $I$.
Then, $\forall I \in  \langle d\rangle^m$ we define the multi-index $\vphi(I) \in S^{(m)}$ 
as 
\linebreak
$\vphi(I)=(\vphi_1(I),...,\vphi_d(I))$. We observe that $\vphi(I)$ is \bea{well defined}, since 
\linebreak
$\sum_{k=1}^{d} \vphi_k(I)=m$, for any $I  \in  \langle d\rangle^m$.
\end{remark}

For the computation of the Frobenius norm of a symmetric tensor it will be useful to introduce the following index classes:
\begin{definition}
\label{remark2}
For a fixed $I \in \langle d\rangle^m$ we define the \textit{\bea{equivalence} class of $I$ under the action of $\vphi$} as
$$
[I]_{\vphi}:=\{ J \in \langle d\rangle^m: \ \vphi(I)=\vphi(J) \}\ ,
$$
and the set of classes 
$$
\langle d\rangle^m/ \vphi := \{ [I]_{\vphi}: \ I \in \langle d\rangle^m \}\ .
$$
\end{definition}
It is easy to show that there is a bijection between the quotient set $\langle d\rangle^m/\vphi$ and $S^{(m)}$ through the identification $[I]_\vphi \subset \langle d\rangle^m$ and $\alpha=\vphi(I)$, for each $\alpha \in S^{(m)}$.
We observe that:
\begin{itemize}
\item If $\vphi(I)=\alpha=(\alpha_1,...,\alpha_d)$, then $[I]_{\vphi}$ has exactly $\gamma_{\alpha}=\frac{m!}{\alpha_1 ! \cdots \alpha_d !}$ elements.
\item If $D=(D_{I})_{I \in \langle d\rangle^m}$ is symmetric, then $D_{I}=D_{J}$ if $I$ and $J$ are in the same class.
\end{itemize}
We will use these two properties in the proof of Proposition \ref{prop-FrobD}, for example to compute the Frobenius norm of a symmetric tensor. 

\begin{definition}
\label{defsymmalpha}
Let $D=(D_{I})$ be a symmetric $m$-tensor and $I \in \langle d\rangle^m$. Then, for any $\alpha=(\alpha_1,...,\alpha_d) \in S^{(m)}$ we define
$$
D_{\alpha}:=D_{I}, \quad \text{if } \alpha=(\vphi_1(I),...,\vphi_d(I)).
$$
We observe that this notion is well-defined since $D$ is symmetric and the property $\vphi(I)=\vphi(\sigma(I))$ holds.
\end{definition}
The previous definition shows that \bea{$\vphi$ induces} a one-to-one correspondence between the indices of a symmetric $m$-tensor and the elements of $S^{(m)}$. This implies that the dimension of $F^{(m)}$ is equal to the cardinality of $S^{(m)}$, i.e. $\Gamma_m$ \bea{(see \eqref{eq:Gamma_m})}.
Hence, for defining $D \in F^{(m)}$ we just need to define $D_{\alpha}$ for every $\alpha \in S^{(m)}$. 

Next we define the order-$m$ outer product and discuss the rank-1 decomposition of tensors, using a result from 
\bea{multilinear algebra (\cite{CGLM}, Lem\-ma 4.2)}.
\begin{definition}
Let $v_i:=(v^{(i)}_1,...,v^{(i)}_d)$, $i=1,...,m$ be $m$ vectors in $\R^d$. We define $v_1\otimes \cdots  \otimes v_m \in T^{(m)}$ as the $m$-tensor with components 
$$
(v_1 \otimes \cdots \otimes v_m)_{I}:= v^{(1)}_{i_1} \cdots v^{(m)}_{i_m}, \quad \forall I=(i_1,...,i_m) \in \langle d\rangle^m.
$$
We call this operation between $m$ vectors, \textit{$m$-outer product}.

In the special case of all the vectors $v_i=v \in \R^d$, $i=1,...,m$ equal, we denote 
$$
v^{\otimes m}:=v\otimes \cdots \otimes v,
$$
and we observe that the tensor $v^{\otimes m}$ is symmetric by definition.
\end{definition}
\begin{prop}[\cite{CGLM}, Lemma 4.2]
\label{TEOdec}
Let $D \in F^{(m)}_d$. Then, there exist an integer $s\in \bea{\langle \Gamma_m \rangle}$, numbers $\lambda_1,...,\lambda_s \in \R$, and vectors $v_1,...,v_s \in \R^d$ such that
\begin{equation}
\label{decompD}
D=\sum_{k=1}^s \lambda_k v_k^{\otimes m}.
\end{equation}
The minimum $s$ such that \eqref{decompD} holds is called \textit{the symmetric rank of $D$}.
\end{prop}
\begin{remark}
In \cite{CGLM} the result is stated for complex tensors. In that case it is possible to choose all the coefficients $\lambda_i$ in \eqref{decompD} equal to one, due to the fact that $\C$ is a closed field.  We remark that the same decomposition carries over to the real case, i.e. with real coefficients $\lambda_i$ and real vectors $v_i$, by using the same proof \cite{CGLMpc}.
\end{remark}
It is easy to see that this rank-1 decomposition persists under a (constant) multilinear matrix multiplication:
\begin{lemma}
\label{fact1}
\bea{Let $B \in \R^{d \times d}$. For any $D \in F^{(m)}_d$ decomposed as in formula \eqref{decompD}, the following decomposition holds:}
\begin{equation}
\label{decomwithB}
B \odot^m D = \sum_{k=1}^{s} \lambda_k (B v_k)^{\otimes^m}.
\end{equation}
\end{lemma}

For rank-1 tensors, their inner product simplifies as follows:
\begin{lemma}
\label{fact2}
Given $v_k=(v^{(k)}_i) \in \R^d$, $k=1,...,2m$, then
\begin{equation}
\label{eqfact2}
\langle v_1 \otimes \cdots \otimes v_m, v_{m+1} \otimes \cdots \otimes v_{2m} \rangle _{\mathcal{\F}}=  \prod_{i=1}^m \langle v_i, v_{i+m} \rangle,
\end{equation}
where $\langle v_i,v_j \rangle$ is the inner product in $\R^d$.
\end{lemma}

A special case of this lemma is given by
\begin{corollary}
\label{corfact2}
Given $v_1,v_2 \in \R^d$, then
\begin{equation}
\label{eqcorfact2}
\langle v_1^{\otimes^m}, v_2^{\otimes^m} \rangle _{\mathcal{F}}=  \langle v_1,v_2 \rangle ^m.
\end{equation}
\end{corollary}

Next we shall derive some results on matrix-tensor products $B\odot^k A$:
\begin{lemma}
\label{step2}
Let $B=B^T \in \R^{d \times d}$ be such that $B \geq 0$. Then, for any $A \in T^{(m)}$
\begin{equation}
\label{step2eq}
\langle A,  B \odot A \rangle _{\mathcal{F}} \geq 0.
\end{equation}
\end{lemma}
For $B \in \R^{d \times d}$, $\|B\|$ we will denote in the sequel the \textit{spectral norm of B}.
\begin{lemma}
\label{LEM4}
For any $A \in T^{(m)}_d$, $B \in \R^{d \times d}$ and $1\leq k \leq m$,
\begin{equation}
\label{EQLEM4}
 \|B \odot ^k A \|_{\mathcal{F}} \leq \| B\|^{k} \|A\| _{\mathcal{F}}.
\end{equation}
\end{lemma}

\subsection{Time evolution of the tensors $D^{(m)}(t)$ in $V^{(m)}$}
Proposition \ref{formulaevolVm} gives the time evolution of each vector $d^{(m)}$. But for $m\ge2$ it does not reveal its inherent structure. Therefore we shall now regroup the elements of $d^{(m)}$ as an order-$m$ tensor and analyze its evolution. 
\begin{definition}
Let $m \geq 1$, $t \geq 0$, and $d^{(m)}(t)=(d_{\alpha}(t))_{\alpha \in S^{(m)}}\in \R^{\Gamma_m}$ be the solution of the ODE $\frac{d}{dt}d^{(m)}=-C^{(m)} d^{(m)}$, \bea{with the matrix $C^{(m)}$ discussed in Remark \ref{matrix-Cm}.} Then we define the symmetric $m$-tensor $D^{(m)}(t)=(D^{(m)}_{\alpha}(t))_{\alpha \in S^{(m)}}$ as 
\begin{equation}
\label{defDm}
D^{(m)}_{\alpha}(t):= \frac{d_{\alpha}(t)}{\gamma_{\alpha}},
\end{equation}
where $\gamma_{\alpha}:= \frac{m!}{\alpha!}$, for $\alpha=(\alpha_1,...,\alpha_d)$.
\end{definition}
For $m=1$ we of course have $D^{(1)}=d^{(1)}\bea{=(d_{\alpha})_{\alpha \in \langle d \rangle}}$.
We illustrate \bea{the above} definition for the case $m=d=2$ with $\Gamma_2=3$:
$$
d^{(2)}=\left( \begin{matrix}
d_{(2,0)}
\\
d_{(1,1)}
\\
d_{(0,2)}
\end{matrix}
\right ), \qquad 
D^{(2)}= \left( \begin{matrix}
d_{(2,0)} & \frac{d_{(1,1)}}{2}
\\
\frac{d_{(1,1)}}{2} & d_{(0,2)}
\end{matrix}
\right) \in F_2^{(2)} \subset T_2^{(2)} =\R^{2 \times 2}.
$$
Elementwise, the evolution of $D^{(m)}_{\alpha}$ easily carries over from Proposition \ref{formulaevolVm}:
\begin{prop}
\label{LEevDt}
For any $\alpha \in S^{(m)}$, the element $D^{(m)}_{\alpha}(t)$ evolves according to
\begin{equation}
\label{claimDm}
\bea{\frac{d}{dt} D_{\alpha}^{(m)}}= - \sum_{j,l=1}^d \alpha_j C_{jl} D^{(m)}_{(\alpha^{(j-)})^{(l+)}} \,.
\end{equation}
\end{prop}
\begin{proof}
From \eqref{evolutiondm} we obtain by substituting the definition \eqref{defDm} on both sides:
\begin{equation}
\label{EQevDt}
\bea{\frac{d}{dt} D^{(m)}_\alpha } = -\frac{1}{\gamma_{\alpha}} \sum_{j,l=1}^d \mathbb{1}_{\alpha_j\ge 1} \gamma_{ (\alpha^{(j-)})^{(l+)}}(\alpha^{(j-)})^{(l+)}_l C_{jl} D^{(m)}_{ (\alpha^{(j-)})^{(l+)}} \,.
\end{equation}
The claim \eqref{claimDm} then follows from the relation
\begin{equation}
\label{eqtec2}
\gamma_{\alpha} \alpha_j= \gamma_{(\alpha^{(j-)})^{(l+)}} (\alpha^{(j-)})^{(l+)}_l \qquad 
\forall \alpha \in \N_0^d \mbox{ with } \alpha_j\ge 1 \,, 
\end{equation}
which can be obtained as follows: It is trivial for $l=j$, and for $l\ne j$ it follows from the definition of $\gamma_{\alpha}$
and from the observation that $(\alpha^{(j-)})^{(l+)}_l=\alpha_l+1$ and $(\alpha^{(j-)})^{(l+)}_j=\alpha_j-1$.
\end{proof}
The advantage of this new structure consists in two facts:
\begin{itemize}
\item The Frobenius norm $\|D^{(m)}(t)\|_{\mathcal{F}}$ is proportional (uniformly in $t$) to the Euclidean norm 
$\left\|\tilde{d}^{(m)}(t)\right\|_2$ for which we want to prove a decay estimate like  \eqref{d1}.
\item The rank-1 decomposition of $D^{(m)}(t)$ is compatible with the Fokker-Planck flow in $V^{(m)}$, \bea{i.e.}, for each symmetric tensor $D^{(m)}(0)$ (considered as an initial condition in $V^{(m)}$), we can decompose $D^{(m)}(t)$ as a sum of order-$m$ outer products of vectors that are solutions of the ODE $\frac{d}{dt} v(t)=-Cv(t)$.
\end{itemize}

Concerning the first property we have
\begin{prop}\label{prop-FrobD}
Given $m \geq 1$, then 
\begin{equation}
\label{normofDmlink}
\left\|D^{(m)}(t)\right\|_{\mathcal{F}}=\frac{1}{\sqrt{m!}} \left\|\tilde{d}^{(m)}(t)\right\|_2 \,, \quad \forall t \geq 0 \,.
\end{equation}
\end{prop}
\begin{proof}
We compute, using Remark \ref{remark2},
$$
\|D^{(m)}(t)\|^2_{\mathcal{F}}= \sum_{I \in \langle d\rangle^m} D^{(m)}_I(t)^2
= \sum_{\alpha \in S^{(m)}} D^{(m)}_{\alpha}(t)^2 \gamma_\alpha,
$$
where we used the identification $D_{\alpha}^{(m)}(t):=D_{I}^{(m)}(t)$ if $\alpha=\vphi(I)$ as well as $\left|[I]_\vphi\right| = \gamma_\alpha$.

Then, using the definition of $D^{(m)}(t)$, $\tilde{d}_{\alpha}(t)=\|g_{\alpha}\|_{\mathcal{H}} d_{\alpha}(t)$, and Lemma \ref{normofgalpha}, we have
\begin{align*}
\label{Dmdmnorm}
\left\|D^{(m)}(t)\right\|^2_{\mathcal{F}}=& \sum_{\alpha \in S^{(m)}} \frac{d_{\alpha}(t)^2}{\gamma_{\alpha} }= 
\sum_{\alpha \in S^{(m)}} \frac{\tilde{d}_{\alpha}(t)^2}{\gamma_{\alpha}\|g_{\alpha}\|^2_{\mathcal{H}}}= \frac{1}{m!} \sum_{\alpha \in S^{(m)}}\tilde{d}_{\alpha}(t)^2 \\
=& \frac{1}{m!} \left\|\tilde{d}^{(m)}(t)\right\|_2^2,
\end{align*}
concluding the proof.
\end{proof}

Concerning the second property we find that the rank-1 decomposition of $D^{(m)}(t)$ commutes with the time evolution by the Fokker-Planck equation:

\begin{theorem}
\label{evolutionDmt}
Let $m\geq 1$ be fixed and let $D^{(m)} \in F^{(m)}$, having the rank-1 decomposition $D^{(m)}= \sum_{k=1}^s \lambda_k v_k^{\otimes  m}$ with symmetric rank $s$, constants $\lambda_1,...,\lambda_s \in \R$ and $s$ vectors $v_k:=(v^{(k)}_j)_{j =1}^d \in \R^d$.
Then, $D^{(m)}(t)$, $t>0$, the solution to \eqref{claimDm} with initial condition $D^{(m)}(0)=D^{(m)}$ has the decomposition
\begin{equation}
\label{EQdecDt}
D^{(m)}(t)= \sum_{k=1}^s \lambda_k [v_k(t)]^{\otimes m},
\end{equation}
where all vectors $v_k(t)\in \R^d$, $k=1,...,s$ satisfy the ODE $\frac{d}{dt}v_k(t)=-Cv_k(t)$ with initial condition $v_k(0)=v_k$. Moreover, $D^{(m)}(t)$, $t>0$ has the constant-in-$t$ symmetric rank $s$.
\end{theorem}
\begin{proof}
We shall compute the evolution of the symmetric $m$-tensor $A(t):=\sum_{k=1}^s \lambda_k [v_k(t)]^{\otimes m}$, using that $\frac{d}{dt}v_k(t)=-Cv_k(t)$. To this end we compute first the derivative $\frac{d}{dt} (w(t)^{\otimes m})_{\alpha}$ if the vector
$w (t)=(w_1(t),...,w_d(t))^T\in \R^d$ satisfies the ODE \bea{$\frac{d}{dt}w(t)=-Cw(t)$}.

Given $\alpha=(\alpha_1,...,\alpha_d)\in S^{(m)}$, we have
\begin{align*}
\frac{d}{dt}(w(t)^{\otimes m})_{\alpha}&=\frac{d}{dt}\prod_{j=1}^d w_j(t)^{\alpha_j}
\\
&= \sum_{j=1}^d \alpha_j \left (w_1(t)^{\alpha_1}\cdots w_j(t)^{\alpha_j -1} \cdots w_d(t)^{\alpha_d} \right ) \left(\frac{d}{dt}w_j(t)\right )
\\
&=-\sum_{j=1}^d \alpha_j \left (w_1(t)^{\alpha_1}\cdots w_j(t)^{\alpha_j -1} \cdots w_d(t)^{\alpha_d} \right ) \sum_{l=1}^d C_{jl}w_l(t)
\\
&=-\sum_{j,l=1}^d \alpha_j C_{jl} \left (w_1(t)^{\alpha_1}\cdots w_j(t)^{\alpha_j -1} \cdots w_l(t)^{\alpha_l+1} \cdots w_d(t)^{\alpha_d} \right )
\\
&=- \sum_{j,l=1}^d \alpha_j C_{jl}\left(w(t)^{\otimes m} \right)_{(\alpha^{(j-)})^{(l+)}},
\end{align*}
and hence, by linearity
\begin{equation}
\label{EQderA}
\frac{d}{dt}\left (A(t) \right )_{\alpha}=- \sum_{j,l=1}^d \alpha_j C_{jl}  \left ( A(t) \right ) _{(\alpha^{(j-)})^{(l+)}}.
\end{equation}
This ODE equals the evolution equation \eqref{claimDm} for $D^{(m)}$, and hence $A(t)=D^{(m)}(t)$ follows.
\\
Next we consider the symmetric rank of $D^{(m)}(t)$, $t>0$. If it would be smaller than $s$, a reversed evolution to $t=0$ would lead to a contradiction to the symmetric rank of $D^{(m)}$.
\end{proof}
This theorem allows to reduce the evolution of the tensors $D^{(m)}(t)$ to the ODE for the vectors $v_k(t)$. This will be a key ingredient for proving sharp decay estimates of $D^{(m)}$ in the next section.
Moreover it provides a compact formula for the evolution of $D^{(m)}(t)$.
\begin{corollary}
Let $m\geq 1$ be fixed. Then, $D^{(m)}(t)$, t>0, the solution to \eqref{claimDm} follows the evolution
\begin{equation}
\frac{d}{dt}D^{(m)}(t)=-m \ \mathrm{Sym}{(C \odot D^{(m)}(t))}, \quad t>0.
\end{equation}
\end{corollary}
\begin{proof}
We shall use the decomposition \eqref{EQdecDt} for $D^{(m)}(t)$. First, we compute the evolution of $[v(t)]^{\otimes m}$, if $\frac{d}{dt}v(t)=-Cv(t)$:
\begin{align*}
\frac{d}{dt}([v(t)]^{\otimes m})&=-\sum_{k=0}^{m-1} [v(t)]^{\otimes k}\otimes ((Cv(t)) \otimes [v(t)]^{\otimes (m-k-1)}\\
&=-m \ \mathrm{Sym}{ \left( (Cv(t)) \otimes [v(t)]^{\otimes(m-1)} \right) }.
\end{align*}
In the last equality we have used, with $w:=Cv(t)$, the general formula 
$$
\mathrm{Sym}{(w \otimes v^{\otimes(m-1)})}=\frac{1}{m} \sum_{k=0}^{m-1}(v^{\otimes k} \otimes w \otimes v^{\otimes (m-k-1)}), \quad \forall v,w \in \R^{d}
$$
that can be proven with a straightforward computation.
By using the linearity of $\mathrm{Sym}$ in $T^{(m)}$, we obtain
\begin{align*}
\frac{d}{dt}D^{(m)}(t)&=\frac{d}{dt} \sum_{k=1}^s \lambda_k [v_k(t)]^{\otimes m}=-m \  \left ( \sum_{k=1}^s \lambda_k  \mathrm{Sym}{\left( (Cv_k(t)) \otimes [v_k(t)]^{\otimes(m-1)} \right)}  \right )
\\
&=-m \ \mathrm{Sym}{ \left( \sum_{k=1}^s \lambda_k (Cv_k(t)) \otimes [v_k(t)]^{\otimes(m-1)}  \right)}=-m \  \mathrm{Sym} (C \odot D^{(m)}(t)).
\end{align*}
\end{proof}

\section{Decay of the subspace evolution in $V^{(m)}$}
\label{sec:fifth}
First we shall rewrite our main decay result, Theorem \ref{maintheo} in terms of tensors for all subspaces $V^{(m)}$. We recall $h(t):=\left\|e^{-Ct}\right\|_{\mathcal{B}(\R^d)}$, which satisfies
\begin{equation}
\label{h1}
h(t) \leq 1 \,, \qquad t \geq 0 \,.
\end{equation}
This follows from 
\begin{equation*}
\frac{d}{dt}\left\|e^{-Ct}x_0\right\|_2^2=-2\langle C_S x(t), x(t) \rangle  \leq 0 \,, \qquad  x_0 \in \R^d, \, 
\end{equation*}
\bea{for $x(t)=e^{-Ct}x_0.$}
\bea{
Using Theorem \ref{shortTimeODE}, the statement of \eqref{h1} can be improved immediately to 
\begin{equation}\label{h1'}
  h(t)<1,\qquad t>0.
\end{equation}
}
We have shown in \eqref{d1} that the inequality \eqref{equivestimate1}, see below, holds with $m=1$, since $D^{(1)}(t)=d^{(1)}(t)$ satisfies the evolution $\dot d^{(1)}=-C d^{(1)}$. Next we extend the estimate \eqref{equivestimate1} to general $m \geq1$. To this end we will show in the next theorem that the propagator norm in each $V^{(m)}$ is the $m$-th power of the propagator norm of the ODE $\dot x=-C x$. This will be used to derive the decay estimates for 
$\left\|e^{-Lt}\right\|_{\bea{\mathcal{B}(V_0^{\perp})}}$.
\begin{theorem}
\label{TEOnormm}
For each $m \geq 1$, $D^{(m)}(0) \in F^{(m)}$, and $D^{(m)}(t)$ defined as in \eqref{defDm}, the following estimate holds:
\begin{equation}
\label{EQTEOnormm}
\left\|D^{(m)}(t)\right\|_{\mathcal{F}} \leq h(t)^m \left\|D^{(m)}(0)\right\|_{\mathcal{F}} \,, \qquad t \geq 0 \,.
\end{equation}
Moreover,
\begin{equation}
\label{supDm}
\sup_{0 \neq D^{(m)}(0) \in F^{(m)}} \frac{\|D^{(m)}(t)\|_{\mathcal{F}}}{\|D^{(m)}(0)\|_{\mathcal{F}}}=h(t)^m.
\end{equation}
\end{theorem}
\begin{proof}
Given the initial condition $D^{(m)}(0) \in F^{(m)}$, Theorem \ref{evolutionDmt} provides its rank-1 decomposition as
\begin{equation}
\label{EQ1641}
D^{(m)}(t)=\sum_{k=1}^s \lambda_k [v_k(t)]^{\otimes m}= \sum_{k=1}^s \lambda_k [e^{-Ct}v_k]^{\otimes m}=e^{-Ct} \odot^m D^{(m)}(0), \quad \forall t \geq 0,
\end{equation}
with $v_k(t)=e^{-Ct}v_k$, for $k=1,...,s$, where we have used Lemma \ref{fact1} in the last equality.
Using \eqref{EQLEM4} then yields:
\begin{equation}\label{D-est}
\|D^{(m)}(t) \|_{\mathcal{F}}=\|e^{-Ct} \odot^m D^{(m)}(0)\|_{\mathcal{F}}\leq \|e^{-Ct}\|^m \|D^{(m)}(0)\|_{\mathcal{F}},
\end{equation}
proving \eqref{EQTEOnormm}.

In order to prove the equality \eqref{supDm} we choose initial data of the form $D^{(m)}(0):=v^{\otimes m}$, $v \in \R^d$. In this case the Frobenius norm factorizes, i.e. \linebreak
$\|D^{(m)}(0)\|_{\mathcal{F}}=\|v\|_2^m$ and 
$$
\| D^{(m)}(t)\|_{\mathcal{F}}=\| (e^{-Ct}v)^{\otimes m }\|_{\mathcal{F}}=\|e^{-Ct}v\|^m_2
$$
We conclude by observing that
$$
\sup_{0 \neq v \in \R^d} \frac{\|e^{-Ct}v\|_2^m}{\|v\|_2^m}=h(t)^m.
$$
\end{proof}

The key step in the above proof is to write the evolution of the tensor $D^{(m)}(t)$ as in \eqref{EQ1641}, which allows for the simple estimate \eqref{D-est}. In contrast, using the rank-1 decomposition in $\|D^{(m)}(t)\|_{\mathcal{F}}^2$ would not be helpful, since the vectors $v_k(t)$ are in general not orthogonal.

We conclude this chapter with the proof of our main result, Theorem \ref{maintheo}, by using Theorem \ref{TEOnormm}.
\begin{proof}[Proof of Theorem \ref{maintheo}]
The first step consists in proving the inequality 
\begin{equation}
\label{FinalIn}
\left\|e^{-Lt}\right\|_{\bea{\mathcal{B}(V_0^{\perp})}} \leq h(t), \forall t \geq 0.
\end{equation}
We can derive the estimate \eqref{FinalIn} from the same ones that hold for the tensors $D^{(m)}(t)$ at each level $m$. More precisely, \eqref{FinalIn} holds if
\begin{equation}
\label{equivestimate1}
\|D^{(m)}(t)\|_{\mathcal{F} }\leq h(t) \|D^{(m)}(0)\|_{\mathcal{F}}, \qquad  t \geq 0, \quad  D^{(m)}(0) \in F^{(m)}, \quad m \geq 1,
\end{equation}
where $D^{(m)}(t)$ is defined as in \eqref{defDm}.
Indeed,
\begin{equation}
\label{linkfandD}
\|f(t)-f_{\infty}\|^2_{\mathcal{H}}= \sum_{m \geq 1}\| \Pi_m  f(t)\|_{\mathcal{H}}^2=\sum_{m \geq 1}\|\tilde{d}^{(m)}(t)\|_2^2=\sum_{m \geq 1} m ! \ \| D^{(m)}(t)\|_{\mathcal{F}}^2, \quad t \geq 0,
\end{equation}
where we have used the orthonormal decomposition of $f(t)$, formulas \eqref{decompoff}, \eqref{normofDmlink},  and that the coefficient $d_0 (t) \equiv 1$ (with the index $0\in \N_0^d$) is constant in time, since $Lg_0=0$ and the normalization $\int_{\R^d}f_0 dx =1$.
Let us assume \eqref{equivestimate1}. Then, 
\begin{align*}
\|f(t)-f_{\infty}\|^2_{\mathcal{H}} =& \sum_{m \geq 1} m !  \ \|D^{(m)}(t)\|^2_{\mathcal{F}} \leq h(t)^2 \sum_{m \geq 1} m ! \  \|D^{(m)}(0)\|_{\mathcal{F}}^2
\\
=& h(t)^2\|f_0-f_{\infty}\|^2_{\mathcal{H}},
\end{align*}
proving \eqref{FinalIn}.

Next, the proof of \eqref{equivestimate1} is a direct consequence of Theorem \ref{TEOnormm} and $h(t) \leq 1$, yielding
$$
\|D^{(m)}(t)\|_{\mathcal{F}} \leq (h(t) )^m \|D^{(m)}(0)\|_{\mathcal{F}} \leq h(t) \|D^{(m)}(0)\|_{\mathcal{F}}.
$$

Now that \eqref{FinalIn} has been proved, we need to show that it is actually an equality, in order to conclude the proof of \eqref{mainequality}. 
For this purpose, we observe that for $m=1$, $D^{(1)} \in \R^d$ evolves according to the ODE $\dot{x}=-Cx$ (see \eqref{d1ev}). Then, it is sufficient to choose an initial datum $f_0 \in V^{(1)}$ to achieve the equality, concluding the proof.
\end{proof}

\bea{
\begin{remark}
Using \eqref{h1'}, the decay estimates \eqref{EQTEOnormm} show that the higher subspace components $D^{(m)}(t)$ decay, for each fixed $t>0$, with a rate that increases exponentially in $m$. Due to the subspace decomposition \eqref{linkfandD}, 
this enhanced decay of the higher subspace components translates into a parabolic-type regularization of the FP-semigroup for $t>0$, cp.\ to Proposition \ref{RegWitha}.
\end{remark}
}

\section{Second quantization}
In this last section we are going to write the FP-operator $L$ in \eqref{NormalizedFPwithC} in terms of the second quantization formalism. This ``language'' was introduced in quantum mechanics in order to simplify the description and the analysis of quantum many-body systems. The assumption of this construction is the  indistinguishability of particles in quantum mechanics. Indeed, according to the statistics of particles, the exchange of two of them does not affect the status of the configuration, possibly up to a sign. Since we are dealing with symmetric tensors, we are going to consider the case in which the sign does not change, i.e. the wave function is identical after this exchange. This is the case of particles that are called \textit{bosons}. 

The functional spaces of second quantization are the so-called \textit{Fock spaces}, that we are going to define in this section. When a single Hilbert space $H$ describes a single particle, then   it is convenient to build an infinite sum of symmetric tensorization of $H$ in order to represent a system of (up to) infinitely many indistinguishable particles, i.e. the Fock space over $H$.  

In the first part of this section the definitions of the Boson Fock space and second quantization operators are given. These constructions will be needed in order to write the FP-operator $L$ as the second quantization of its corresponding drift matrix $C$. This will be the main result of the second part of this section as an application of well known results in the literature.

\begin{subsection}{The Boson Fock space}
In the next definition we will use the notion of \text{$m$-fold tensor product} over a Hilbert space $H$. This is a generalization of the space of order-$m$ hypermatrices $T^{(m)}$ defined in $\S 5$, where the Hilbert space was the finite dimensional space $\R^d$. In the quantum mechanics literature, the role of the Hilbert space is often played by $L^2(\R^3; \mathbb{C})$, in order to describe the wave function of a quantum particle. For a more complete explanation of tensor products of Hilbert spaces and Fock \linebreak spaces we refer to \S\RomanNumeralCaps{2}.4  \ in \cite{RS1}.

In the literature, Fock spaces are mostly considered for Hilbert spaces over the field $\C$. But since the FP-equations \eqref{FP} and \eqref{NormalizedFPwithC} are posed on $\R^d$ (and not over $\C^d$), we shall use here only real valued Fock spaces. Moreover, these FP-equations are considered here only for real valued initial data, and hence real valued solutions.
\begin{definition}
Let $H$ be a Hilbert space and denote by $H^{(m)}:=H \otimes H\otimes \cdots  \otimes H$ ($m$ times), for any \bea{$m \geq 1$}. Set ${H}^{(0)}:=\C$ (or $\R$) and define the \textit{Fock space over $H$} as the completed direct sum
\begin{equation}
\label{defFock}
\mathcal{F}(H)= \bigoplus_{m=0}^{\infty} H^{(m)}.
\end{equation}
Then, an element $\psi \in \mathcal{F}(H)$ can be represented as a sequence $\psi=\{ \psi^{(m)} \}_{m=0}^{\infty}$, where $\psi^{(0)} \in \C$ (or $\R$), $\psi^{(m)} \in H^{(m)}, \forall \bea{m \geq 1}$, so that 
\begin{equation}
\|\psi\|_{\mathcal{F}(H)}:=\sqrt{\sum_{m=0}^{\infty} \|\psi^{(m)}\|^2_{H^{(m)}} }< \infty.
\end{equation}
Here $\|\cdot\|_{H^{(m)}}$ denotes the norm induced by the inner product in $H^{(m)}$ (see Proposition $1$, \S \RomanNumeralCaps{2}.4  in \cite{RS1}). 
\end{definition}

As we anticipated, we will rather work with a subspace of $\Fock$, the so-called Boson Fock space that we are going to define. First we need to define the $m$-fold symmetric tensor product of $H$ as follows:

Let $\mathcal{P}_m$ be the permutation group on $m$ elements and let $\{\phi_k\}$; $k=1,...,\dim H$, be a basis for $H$. For each $\sigma \in \mathcal{P}_m$, we define its corresponding operator (we will still denote it with $\sigma$) acting on basis elements of $H^{(m)}$ by
\begin{equation}
\label{7star}
\sigma( \phi_{k_1} \otimes \phi_{k_2} \otimes \cdots \otimes \phi_{k_m}):= \phi_{\sigma(k_1)} \otimes \phi_{\sigma{(k_2)}} \otimes  \cdots \otimes \phi_{\sigma(k_m)}.
\end{equation}
Then $\sigma$ extends by linearity to a bounded operator on $H^{(m)}$. 
With the previous definition \eqref{7star} we can define the operator $S_m:=\frac{1}{m!}\sum_{\sigma \in \mathcal{P}_m} \sigma$ that acts on $H^{(m)}$. Its range $S_m H^{(m)}$ is called the \textit{$m$-fold symmetric tensor product of $H$}.  Let us see examples of $S_m H^{(m)}$.
\begin{example}
Let us consider first the case $H=L^2(\R)$ and $H^{(m)}=L^2(\R) \otimes \cdots \otimes L^2(\R)$. Since $H^{(m)}$ is isomorphic to $L^2(\R^m)$, it follows that an element $\psi^{(m)} \in S_m H^{(m)}$ is a function $\psi^{(m)}(x_1,...,x_m)$ in $L^2(\R^m)$ left invariant under any permutation of the variables. It is used in quantum mechanics to describe the quantum states of $m$ particles that are not distinguishable.

For our purposes, we will deal with $H=\R^d$. In this case it is easy to check that $S_m H^{(m)}$ corresponds to the space of symmetric $m$-tensors $F^{(m)}$ that we defined in $\S 5$, equipped with the Frobenius norm. $\hfill  \qed$
\end{example}
\begin{definition}
The subspace of $\mathcal{F}(H),$
\begin{equation}
\mathcal{F}_s(H):= \bigoplus_{m=0}^{\infty} S_m H^{(m)}
\end{equation}
 is called the \textit{symmetric Fock space over} $H$ or the \textit{Boson Fock space over} $H$.
\end{definition}
\end{subsection}

\begin{subsection}{The second quantization operator}
In order to write the \bea{FP-propa\-gator} in terms of the second quantization formalism, we need to define the second quantization operators (see \S\RomanNumeralCaps{1}.4 in \cite{S} and \S \RomanNumeralCaps{10}.7 in \cite{RS}) acting on the Boson Fock space. 

Let $H$ be a Hilbert space and $\mathcal{F}_s(H)$ be the Boson Fock space over $H$. Let $A$ be a contraction on $H$, i.e., a linear transform of norm smaller than or equal to $1$. Then there is a unique contraction (Corollary \RomanNumeralCaps{1}.15, \cite{S}) $\Gamma(A)$ on $\mathcal{F}_s(H)$ so that
\begin{equation}
\label{GammaA}
\Gamma(A) \restriction_{S_m H^{(m)}} = A \otimes \cdots \otimes A \qquad \text{($m$ times)},
\end{equation}
where the operator $A \otimes \cdots \otimes A$ is defined on each basis element $\psi^{(m)}= \psi_{i_1} \otimes \cdots \otimes  \psi_{i_m}$ of $S_m H^{(m)}$ as 
$$
(A \otimes \cdots \otimes A)(\psi^{(m)}):= (A\psi_{i_1}) \otimes \cdots \otimes (A \psi_{i_m}),
$$
and equal to the identity when restricted to $H^{(0)}$.
In order to prove the above existence of $\Gamma(A)$, the estimate $\|\Gamma(A)\restriction_{S_m H^{(m)}}\| \leq \|A\|^m$ is first \linebreak showed in \cite{S}. This allows to extend the operator $\Gamma(A)$ to the Boson Fock  space by continuity, and by remaining a contraction.
In the case $A=e^{-Ct}$ and $H=\R^d$, the operator $\Gamma(A)$ will be useful to show the link between the Fokker-Planck solution operator $e^{-Lt}$ and the second quantization operators, defined in the following way:
\begin{definition}
Let $H$ be a Hilbert space. Let $A$ be an operator on $H$ (with domain $G(A)$). The operator $d\Gamma(A)$ is defined as follows: Let $G_m(A) \subseteq S_m H ^{(m)}$ be $G(A) \otimes \cdots \otimes G(A)$ and $G(d\Gamma (A)):= \foo_{m=0}^{\infty} G_m(A)$ (incomplete direct sum):
\begin{equation}
\label{defdgamma}
d\Gamma(A) \restriction_{S_m H^{(m)}}:=A \otimes \mathbb{1} \otimes  \cdots \otimes \mathbb{1} + \cdots + \mathbb{1} \otimes \cdots \otimes \mathbb{1} \otimes A, \quad \bea{m \ge 1,} 
\end{equation}
and $d\Gamma (A) \restriction_{H^{(0)}}:=0$.      
The operator $d \Gamma (A)$ is called the \emph{second quantization of $A$}.
\end{definition}
In \cite{S} the following property of the second quantization operator can be found (see \RomanNumeralCaps{1}.41):

Let $A$ generate a $\mathcal{C}_0$-contraction semigroup on $H$. Then the closure of $d\Gamma(A)$ generates a $\mathcal{C}_0$-contraction semigroup on $\mathcal{F}_s(H)$ and 
\begin{equation}
\label{relGamma}
e^{- d\Gamma(A) t }= \Gamma (e^{-At}) \qquad \forall t \geq 0.
\end{equation}
\end{subsection}

\subsection{Application to the operator $e^{-Lt}$}
In the last part of this section we will show that the Fokker-Planck operator $L$ is the second quantization of $C$. First, we shall identify the Hilbert space $L^2(\R^d, f_{\infty}^{-1})$ with a suitable Fock space.

The spectral decomposition and the tensor structure that we introduced in $\S 5$ suggest to consider the Boson Fock space over the finite dimensional Hilbert space $\R^d$, whose elements have components in the space of symmetric tensors $F^{(m)}$. Indeed, we can define an isomorphism $\Psi$ between $L^2(\R^d, f_{\infty}^{-1})$ and $\mathcal{F}_s(\R^d)$ as follows:

Let $f\in L^2(\R^d, f_{\infty}^{-1})$. As we saw in $\S4$, $f$ admits the decomposition $f(x)=\sum_{\bea{m=0}}^{\bea{\infty}} \sum_{\alpha \in S^{(m)}} d_{\alpha} g_{\alpha}(x)$, for \bea{suitable} coefficients $d_{\alpha} \in \R$. For each $m \geq 1$, we define the symmetric tensor $\widetilde{D}^{(m)} \in F^{(m)}$ with components $\widetilde{D}^{(m)}_{\alpha}:= d_{\alpha} \frac{\sqrt{m!}}{\gamma_{\alpha}} \in \R$ (see \eqref{defDm}), $\forall \alpha \in S^{(m)}$. For $m=0$ we choose $\widetilde{D}^{(0)}:=\langle f, f_{\infty} \rangle_{L^2(f_{\infty}^{-1})}$. Hence, by observing that $F^{(m)}= S_m H^{(m)}$, $H:=\R^d$, we define the isometry
\begin{equation}
\Psi: f \in L^2(\R^d, f_{\infty}^{-1}) \rightarrow \psi :=\{\widetilde{D}^{(m)} \}_{m=0}^{\infty} \in \mathcal{F}_s(\R^d).
\end{equation}
It remains to check that $\|\psi \|_{\mathcal{F}_s(\R^d)} < \infty$. This follows from the Plancherel's equality together with \eqref{normofDmlink}. It leads to
$$
\|f\|^2_{L^2(f_{\infty}^{-1})}=\sum_{m=0}^{\infty} \|\widetilde{D}^{(m)}\|^2_{\mathcal{F}}=\|\psi\|^2_{ \mathcal{F}_s(\R^d)}.
$$
Hence, up to an isomorphism, we can consider the FP-operator $L$ also as acting on the Fock space $\mathcal{F}_s(\R^d)$. We conclude the section with the next proposition that allows to write $L$ in the second quantization formalism.
\begin{prop}\label{L2ndQuant}
Let $L$ be the Fokker-Planck operator defined in \eqref{NormalizedFPwithC} and let $C \in \R^{d \times d}$ be its corresponding drift matrix. Then, $L$, now considered as acting on $\mathcal{F}_s(\R^d)$, is the second quantization of $C$, considered as an operator from the Hilbert space $\R^d$ to itself, i.e., $L=d\Gamma(C)$.
\end{prop}
\begin{proof}
Due to the relation \eqref{relGamma}, it is sufficient to prove that the \bea{FP-propa\-gator} $e^{-Lt}$ (considered on $\mathcal{F}_s(\R^d)$) satisfies the equality
\begin{equation}
\label{Lsec}
e^{-Lt}=\Gamma(e^{-Ct}), \qquad \forall t \geq 0.
\end{equation}
\bea{Equivalently,} on each $S_m H^{(m)}$, $m\geq 1$, the formula
\begin{equation}
\label{Lsec2}
e^{-Lt}(\psi^{(m)})=(e^{-Ct} \psi_{i_1}) \otimes \cdots \otimes ( e^{-Ct} \psi_{i_m}), 
\end{equation}
\bea{holds for every basis element $\psi^{(m)}=\bigotimes_{k=1}^m \psi_{i_k}$ of $F^{(m)}$}.

Given an initial condition $f_0 \in L^2(\R^d, f_{\infty}^{-1})$ and its corresponding solution $f(t)=e^{-Lt}f_0$ of \eqref{NormalizedFPwithC}, the isometry $\Psi\bea{: L^2(\R^d, f_{\infty}^{-1}) \rightarrow \mathcal{F}_s(H)}$ maps then \bea{as follows:}
$$\Psi f_0=\psi_0=\{\widetilde{D}^{(m)}(0)\}_{m=0}^{\infty}, \text{ \quad and \quad } \Psi f(t)=\psi(t)=\{ \widetilde{D}(t)^{(m)}\}_{m=0}^{\infty},
$$
respectively.
Then, the factored evolution formula \eqref{EQ1641} for $D^{(m)}(t)=\linebreak \sqrt{m!} \ \widetilde{D}^{(m)}(t)$ proves the equality \eqref{Lsec2}, for each $m \geq 1$. Since the generator of a $\mathcal{C}_0$-semigroup is unique, we obtain $L=d \Gamma (C)$.
\end{proof}
While $C$ is a bounded operator with domain $G(C)=\R^d$, its second quantization $d\Gamma (C)$ is unbounded with dense domain $G(d\Gamma(C)) \subsetneq \mathcal{F}_s(H)$, just like $L$ is unbounded on $L^2(\R^d, f_{\infty}^{-1})$.

Finally, our main result, Theorem \ref{maintheo} reads in the language of second quantization 
\begin{equation}
\Big\|e^{-d\Gamma(C) t} \restriction_{\displaystyle \bigoplus_{\bea{m \geq 1}}S_m H^{(m)}}  \Big\|_{\mathcal{B}(\mathcal{F}_s(H))}=\|e^{-Ct}\|_{\R^{d \times d}}, \qquad t \geq 0.
\end{equation}
Note that the restriction to $\displaystyle\bigoplus_{\bea{m \geq 1}}S_m H^{(m)}$ corresponds to the restriction to $V_0^{\perp}$ in \eqref{mainequality}, the orthogonal of the steady state $f_{\infty}$.

\bea{We remark that Proposition \ref{L2ndQuant} is a special case of Theorem 1 in \cite{CMG}, there formulated for an infinite dimensional Hilbert space setting. We still include a proof here to make this paper self-contained. Moreover, an explicit computation of the spectrum and second quantization formalism for FP-equa\-tions in the infinite dimensional setting were given in \cite{vNe}.
}

\begin{remark}
Many aspects of the above analysis seem to rely importantly on the explicit spectral decomposition of the FP-operator in \S 4.1, i.e. knowing the FP-eigenfunctions (as Hermite functions). We remark that this situation in fact carries over to FP-equations with linear coefficients plus a \emph{nonlocal perturbation} of the form $\theta_f:=\theta*f$ with the function $\theta(x)$ having zero mean, see Lemma 3.8 and Theorem 4.6 in \cite{AS}. For such nonlocally perturbed FP-equations, surprisingly, one still knows all the eigenfunctions as well as its (multi-dimensional) creation and annihilation operators.
\end{remark}
\appendix
\section{Deferred proofs}

\begin{proof}[Proof of Lemma~\ref{fact1}]
We compute the components of the l.h.s.\ of \eqref{decomwithB}. Using \eqref{decompD} with $v_k=(v_i^{(k)})\in\R^d$, we have for any $(i_1,..,i_m) \in \langle d\rangle^m$:
\begin{align*}
(B \odot^m D)_{i_1...i_m}= & \sum _{j_1,...,j_m=1}^d B_{i_1j_1}\cdots B_{i_mj_m} D_{j_1...j_m}=
\sum _{j_1,...,j_m=1}^d B_{i_1j_1}\cdots B_{i_mj_m} \sum_{k=1}^s \lambda_k v^{(k)}_{j_1}\cdots v^{(k)}_{j_m} \\                                         
=& 
\sum_{k=1}^s \lambda_k (Bv_k)_{i_1} \cdots (Bv_k)_{i_m}  
= \left (\sum_{k=1}^s \lambda_k (Bv_k)^{\otimes ^m} \right)_{i_1\cdots i_m},
\end{align*}
concluding the proof.
\end{proof}

\begin{proof}[Proof of Lemma~\ref{fact2}]
By definition,
\begin{align*}
\langle v_1 \otimes \cdots \otimes v_m, v_{m+1} \otimes \cdots \otimes v_{2m} \rangle _{\mathcal{\F}}=
& \sum_{i_1,...,i_m=1}^d (v_1 \otimes \cdots \otimes v_m)_{i_1...i_m} (v_{m+1} \otimes \cdots \otimes v_{2m})_{i_1...i_m}
\\
=& \sum_{i_1,...,i_m=1}^d v^{(1)}_{i_1}\cdots v^{(m)}_{i_m}v^{(m+1)}_{i_1}\cdots v^{(2m)}_{i_{m}}
\\
=& \left( \sum_{i_1=1}^d v^{(1)}_{i_1}v^{(m+1)}_{i_1} \right) \cdots  \left( \sum_{i_m=1}^d v^{(m)}_{i_m}v^{(2m)}_{i_m} \right )
\\
=& \langle v_1,v_{m+1} \rangle \cdots \langle v_m,v_{2m}\rangle .
\end{align*}
\end{proof}

\begin{proof}[Proof of Lemma~\ref{step2}]
We have
\begin{align*}
\langle A, B \odot A \rangle _{\mathcal{F}}= &\sum_{i_1,...,i_m=1}^d A_{i_1...i_m}( B \odot A)_{i_1...i_m}= \sum_{j_1,i_1,...,i_m=1}^d A_{i_1...i_m} B_{i_1j_1} A_{j_1i_2...i_m}
\\
=& \sum_{i_2,...,i_m=1}^{d} \langle x^{(i_2...i_m)}, B x^{(i_2...i_m)} \rangle,
\end{align*}
where, for $i_2,...,i_m$ fixed, $x^{(i_2...i_m)}_{i_1}:=A_{i_1i_2...i_m}$ are vectors in $\R^d$. The claim then follows from $B \geq 0$.
\end{proof}

\begin{proof}[Proof of Lemma~\ref{LEM4}]
First consider the \textit{Case $k=1$.}
We have
\begin{align}
\label{App1}
\|B \odot A \|^2_{\mathcal{F} }= & \sum_{i_1,...,i_m=1}^d (\sum_{j_1=1}^d B_{i_1j_1}A_{j_1i_2...i_m})^2= \sum_{i_2,...,i_m=1}^d \|B x^{(i_2...i_m)}\|^2 
\\
\leq & \sum_{i_2,...,i_m=1}^d \|B\|^2 \|x^{(i_2...i_m)}\|^2  = \|B\|^2 \sum_{i_1,...,i_m=1}^d (x^{(i_2...i_m)}_{i_1})^2 
\\
=& 
\|B\|^2 \|A\|_{\mathcal{F}}^2
\end{align}
where, for $i_2,...,i_m$ fixed, $x^{(i_2...i_m)}_{j_1}:=A_{j_1i_2...i_m}$ are vectors in $\R^d$. 
Note that the estimate \eqref{App1} would hold as well if the matrix-tensor product does not operate on the first index (as in $B \odot A$), but on the $j-$th index, with some $1\leq j \leq m$. Then \eqref{EQLEM4} follows by iterated applications of \eqref{App1}.
\end{proof}

\begin{proof}[Proof of Proposition \ref{Prop319}]
(a) We recall that Theorem \ref{maintheo} and \eqref{h1} imply
\begin{equation*}
\tilde{h}(t)=\|e^{-Lt}\|_{\bea{\mathcal{B}(V_0^{\perp})}}=\|e^{-Ct}\|_2=h(t) \leq 1, \quad t \geq 0.
\end{equation*}
Then, Theorem \ref{TEOnormm} implies \eqref{EQTEOnormm}, $\forall m \geq1$. From \eqref{decompoff} we recall 
\begin{equation}
\label{4stars}
\left| \left| \frac{f(t)}{f_{\infty}} \right| \right|_{L^2(f_{\infty})} ^2=\|f(t)\|_{\mathcal{H}}^2= \sum_{m \in \N_0} \| \tilde d^{(m)}(t)\|^2= \sum_{\beta \in \N_0^d} | \tilde{d}_{\beta}(t)|^2,
\end{equation}
and $ \frac{f(t)}{f_{\infty}}= \sum_{\beta \in \N_0^d} \tilde{d}_{\beta}(t) \hat{g}_{\beta}$, where $\hat{g}_{\beta}:= \frac{\tilde{g}_{\beta}}{f_{\infty}}$ is an orthonormal basis of $L^2(f_{\infty})$.

Using \eqref{hermiteingalpha} and the formula $H^{'}_n(x)=n H_{n-1}(x)$ for Hermite polynomials we compute, for any $\beta \in \N_0^d$,
$$
\partial_{x_j} \hat{g}_{\beta}= \frac{\beta_j H_{\beta_j -1}(x_j)}{\sqrt{\beta!}} \displaystyle \prod_{i \neq j} H_{\beta_i}(x_i), \text{ \  and  \ } \| \partial_{x_j} \hat{g}_{\beta}\|_{L^2(f_{\infty})}=\sqrt{\beta_j},
$$
where we used $\|H_n\|_{L^2(f_{\infty})}=\sqrt{n \ !}$ . This yields, with \eqref{EQTEOnormm} and \eqref{normofDmlink},
\begin{align}
\label{3stars}
\left| \left| \nabla \left(\frac{f(t)}{f_{\infty}} \right) \right| \right|_{L^2(f_{\infty})} ^2=& \sum_{\beta \in \N_0^d} |\tilde{d}_{\beta}(t)|^2 |\beta| = \sum_{m \in \N_0} m \| \tilde{d}^{(m)}(t)\|^2
\\
 \leq &  \sum_{m \in \N_0} m (\tilde{h}(t))^{2m} \|\tilde{d}^{(m)}(0)\|^2, \quad t>0. \nonumber
\end{align}
From the hypothesis on $\tilde{h}$,  we deduce $\tilde{h}(t) \leq 1- c_1 t^{\alpha}$ on $0 \leq t \leq \delta$ for some $0 < c_1 \leq c$ and some $\delta >0$. Then \eqref{3stars} can be estimated further by
$$
\sum_{m \in \N_0} m (1-c_1 t ^{\alpha})^{2m} \| \tilde{d}^{(m)}(0)\|^2 \leq \frac{1}{e c_1} t^{-\alpha} \sum_{m \in \N_0} \| \tilde{d}^{(m)}(0)\|^2, \quad 0 \leq c_1 t^{\alpha} \leq 1.
$$
where we used the elementary inequality $m(1-c_1 t^{\alpha})^{2m} \leq \frac{1}{e c_1}t^{-\alpha}$, $m \in \N_0$.
The main assertion of part (a) then follows from \eqref{4stars}.

Finally we turn to the optimality of $\alpha$: If \eqref{RegH1} would hold for all $f_0\in \mathcal H$ with some $\alpha_1\in(0,\alpha)$, then part (b) of this proposition would imply $\tilde{h}(t)\le 1-c_2t^{\alpha_1}$. But this would contradict the assumption $\tilde{h}(t)=1-ct^\alpha+o(t^\alpha)$. Hence, $\alpha/2$ is indeed the minimal regularization exponent in \eqref{RegH1}.
\\
\\
\medskip (b) For $f_0 \in V^{(m)}, \ m \in \N$ we compute, by using \eqref{3stars} and \eqref{RegH1},
\begin{equation}
\label{step1Proof3.20}
\left| \left| \nabla \left(\frac{f(t)}{f_{\infty}} \right) \right| \right|_{L^2(f_{\infty})} ^2 =m \  \| \tilde{d}^{(m)}(t)\|^2 \leq \tilde{c}^2 t^{-\alpha} \|\tilde{d}^{(m)}(0)\|^2, \qquad 0 < t \leq \delta.
\end{equation}
Then, by taking in \eqref{step1Proof3.20} the supremum w.r.t.\ the set $\{ 0 \neq \tilde{d}^{(m)}(0) \in \R^{\Gamma_m} \}$ and using \eqref{supDm}, \eqref{normofDmlink} we obtain the family of estimates 
\begin{equation}
\label{step2Proof3.20}
  \tilde{h}(t) ^{2m}= \sup_{0 \neq D^{(m)} \in F^{(m)}} \frac{\|D^{(m)}(t)\|^2_{\mathcal{F}}}{\|D^{(m)}\|^2_{\mathcal{F}}}= \sup_{0 \neq \tilde{d}^{(m)}(0) \in \R^{\Gamma_m}} \frac{\|\tilde{d}^{(m)}(t)\|^2}{\|\tilde{d}^{(m)}(0)\|^2}
\leq \frac{\tilde{c}^2}{m} t^{-\alpha}, 
\end{equation}
with $m \in\N,  \ 0 < t \leq \delta$.

Next we will show that this family of estimates for $\tilde{h}(t)$ implies $\tilde{h}(t) \leq 1-c _2t^{\alpha}$ for $0\leq t \leq \delta_2$, with some $c_2>0,  \ \delta _2>0$ (see  Figure \ref{plot2} for the case $\alpha=1$).
For each $m\in\N$ and $t \in I_{\delta}:=(0, \delta]$, we rewrite \eqref{step2Proof3.20} as
\begin{align}
\label{HAndE1}
\tilde{h}(t) \leq  \left ( \frac{\tilde{c}}{\sqrt{m}} t^{-\frac{\alpha}{2}} \right)^{\frac{1}{m}}
=e^{-\frac{1}{2}  \frac{\log (  \bar{c} m t^{\alpha}  )}{m}} =: g(m;t),
\end{align}
with $\bar{c}:=\tilde{c}^{-2}$.
For $t\in I_\delta$ fixed, we now consider the function $g(\mu;t)$ with continuous argument $\mu>0$. $g(\cdot;t)$ has its unique minimum at $\mu_0(t):=\frac{e}{\bar c}t^{-\alpha}$ and it is strictly decreasing on $(0,\mu_0(t))$.

To estimate the minimum of $g$ for the discrete argument $m\in\N$, we consider:
For $0\le t\le t_1:=\big(\frac{e-2}{\bar c}\big)^{1/\alpha}$ we have
$$
  \frac{2}{\bar c}t^{-\alpha} \le \Big\lceil\frac{2}{\bar c}t^{-\alpha}\Big\rceil < \frac{2}{\bar c}t^{-\alpha}+1
  \le \frac{e}{\bar c}t^{-\alpha}=\mu_0(t),
$$
with $\lceil\cdot\rceil$ denoting the ceiling function. We choose the index $m(t):=\big\lceil\frac{2}{\bar c}t^{-\alpha}\big\rceil \in\N$ and use the monotonicity of $g(\cdot;t)$ on $(0,\mu_0(t)]$ to estimate:
$$
 \tilde{h}(t)\leq \min_{m \in \N }{g(m;t)} \leq g(m(t);t) \le g\big(\frac{2}{\bar c}t^{-\alpha};t\big) = e^{-2c_2t^\alpha},
$$
with $c_2:=\frac{\log(2) \bar c}{8}>0$.

With the elementary estimate $e^{-2c_2y}\le 1-c_2y$ on some $[0,t_2]$, we obtain 
$$
  \tilde{h}(t) \leq e^{-2c_2 t^{\alpha} } \leq 1- c_2 t^{\alpha}, \qquad t \in [0, \delta_2],
$$
with $\delta_2:=\min\{t_1,t_2^{1/\alpha}\}$.

Finally we turn to the minimality of $\alpha$: If $\tilde{h}$ would even satisfy the decay estimate $\tilde{h}(t)\le 1-\tilde c_2 t^{\alpha_1}$ with some $\alpha_1\in(0,\alpha)$ and $\tilde c_2>0$, then (the proof of) part (a) of this proposition would imply the regularization estimate \eqref{RegH1} with the exponent $\alpha_1/2$. But this would contradict the assumption on $\alpha$ being minimal in that estimate.
\end{proof}

\begin{figure}[H]
\includegraphics[scale=0.6]{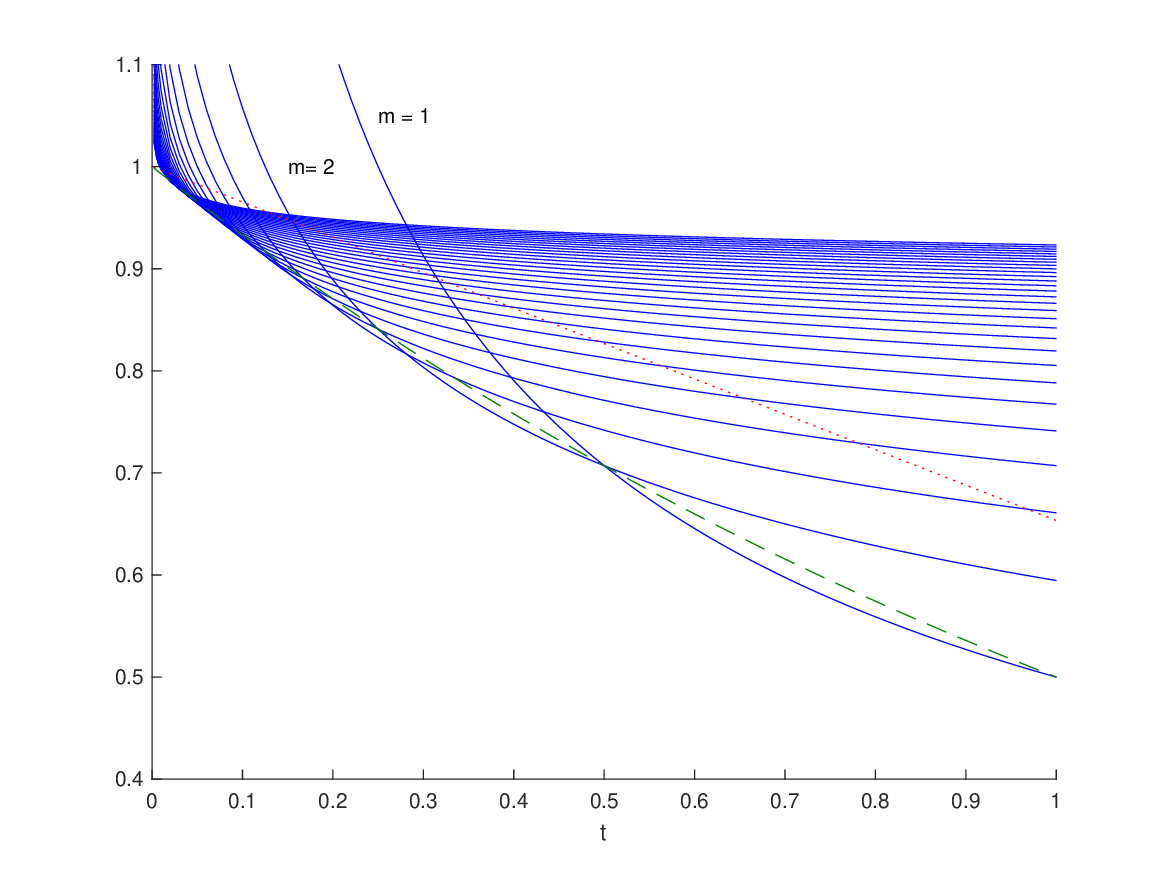}
\caption{The family of decay estimates $h(t) \leq g(m;t)$, $m \in \N$ with $\alpha=1$, $\bar{c}=4$ (solid, blue curves) implies $h(t) \leq e^{-2c_2 t}$, (dashed, green curve), and hence $h(t) \leq 1-c_2t$ (dotted, red line).}
\label{plot2}
\end{figure}

\section*{Acknowledgement}
The authors were partially supported by the FWF (Austrian Science Fund) funded SFB \#F65 and the FWF-doctoral school W 1245. The first author acknowledges fruitful discussions with Miguel Rodrigues that led to Proposition \ref{Prop319}(b), \bea{as well as with Wolfgang Herfort}.

\end{document}